\documentclass{amsart}
\usepackage{amssymb,latexsym,graphics,enumerate,color}
\usepackage{fullpage}
\numberwithin{equation}{section}
\usepackage{bbm}
\usepackage[normalem]{ulem}

\usepackage{xcolor,cancel}

 \usepackage[colorlinks=true]{hyperref}
 \hypersetup{urlcolor=blue, citecolor=blue}

\newtheorem{theorem}{Theorem}[section]
\newtheorem{corollary}[theorem]{Corollary}
\newtheorem{lemma}[theorem]{Lemma}
\newtheorem{proposition}[theorem]{Proposition}

\newcommand{\la}{\langle}
\newcommand{\ra}{\rangle}

\renewcommand{\O}{{\mathcal{O}}}

\newcommand{\R}{{\mathbb{R}}}

\newcommand{\cd}{{\,\cdot\,}}

\newcommand{\ang}{{\not\negmedspace\nabla}}

\newcommand{\tC}{{\tilde{C}}}

\renewcommand{\L}{{\mathcal{L}}}

\renewcommand{\S}{{\mathbb{S}}}
\newcommand{\supp}{{\text{supp }}}

\title[Strauss conjecture with scattering damping]
{The Strauss conjecture for the wave equation with scattering damping}

\author{Elliot Adamson}
\address{Department of Mathematics, University of North Carolina,
  Chapel Hill, NC  27599-3250, USA}
\email{elliotadamsonunc@outlook.com}

\author{Jason Metcalfe}
\address{Department of Mathematics, University of North Carolina,
  Chapel Hill, NC  27599-3250, USA}
\email{metcalfe@email.unc.edu}

\author{Chengbo Wang}
\address{
  School of Mathematical Sciences\\
  Zhejiang Normal University\\
  Jinhua 321004, China
}
\email{wangcbo@zjnu.edu.cn}

\thanks{The first two authors gratefully acknowledge partial support
  from NSF grant 2135998.  The third author was supported by NSFC 12671293.
 }

\begin{document}
\bibliographystyle{plain}

\begin{abstract}
The focus of this article is to establish the existence portion of the
Strauss conjecture for wave equations with spatially dependent scattering
damping.  In the current setting, the damping is not sufficiently
strong to alter the critical Strauss exponent.  In a subsequent study,
the authors will consider the case of damping coefficients with
critical decay where a shift of the Strauss exponent is
expected.  While several recent studies of the Strauss conjecture have
relied upon weighted Strichartz estimates, here, instead, an approach
that is based only on integrated local energy estimates, in particular
those with both $r^p$ and ghost weights, is developed.  This technique
also applies to the problem for the free equation.  Space-time
Klainerman-Sobolev estimates provide the necessary decay to close the argument.
\end{abstract}



\maketitle 

\section{Introduction}

This article focuses on the analog of the Strauss conjecture for wave
equations with spatially dependent scattering damping.  Specifically,
we consider
\begin{equation}\label{mainEquation}
\begin{cases}
  \partial_t^2 w - \Delta w + \frac{a}{(1+|x|)^{1+\delta}} \partial_t
  w = F_p(w),\quad (t,x)\in \R_+\times\R^3,\\
  w(0,\cd)=\varepsilon w_{(0)},\quad \partial_t w(0,\cd)= \varepsilon w_{(1)}
\end{cases}
\end{equation}
where $a, \delta>0$.  The data $w_{(0)}, w_{(1)}$ will be taken to be smooth with
compact support contained within $\{|x|\le \frac{1}{2}\}$.  The value
$1/2$ is chosen to reduce the number of unimportant parameters so as
to increase clarity.  Straightforward modifications will allow for any
finite radius. It is
quite likely that the conditions of smoothness and compact support on the data can be relaxed.  Doing so
would come at the expense of additional notation, while the current
case still contains the primary considerations.  Here, $F_p$ satisfies
\begin{equation}
  \label{Fp}
  \sum_{0\le j\le 2} |w|^j |\partial_w^jF_p(w)|\lesssim |w|^p,\quad
  \text{ when } |w|\ll 1.
\end{equation}

In the case that $a=0$, \eqref{mainEquation} has been widely studied
on $\R_+\times\R^n$.  In short, the Strauss conjecture asserts
(in the $a=0$ case) that \eqref{mainEquation} admits a global solution
for all $\varepsilon>0$ sufficiently small if and only if $p>p_c(n)$
where the critical exponent $p_c(n)>1$ solves
\[(n-1)p_c^2-(n+1)p_c-2=0.\]
For $n=3$, where $p_c(3)=1+\sqrt{2}$, this was first established in
\cite{John79}.  The full existence portion of the conjecture was
proven in \cite{GLS97}, \cite{Ta01-2}, while the blow up for
$p<p_c$(n) was resolved in \cite{Sideris}.

We mention a few recent robust proofs that address the existence portion of
the conjecture in $n=3,4$ in the presence of background geometry: \cite{HMSSZ}, \cite{FaWa}, \cite{LMSTW},
\cite{MW}.  These are rooted in a class of weighted Strichartz
estimates, which are obtained by interpolating between an integrated
local energy estimate and a trace lemma on the sphere.  The current
proof will instead only rely on integrated local energy estimates,
which provides an added degree of robustness to, e.g., absorb
appropriate lower order perturbations.

The conjectured behavior of \eqref{mainEquation} has been discussed in
\cite{Ikehata-Todorova-Yordanov} and
\cite{Nishihara-Sobajima-Wakasugi}.  Additionally, see the summaries
in \cite{lai-tu} and \cite{Lai-Liu-Tu-Wang}.  For $\delta>0$, it
remained the expectation that global existence holds if and only if
$p>1+\sqrt{2}$.  When $-1\le \delta < 0$, the behavior of the heat
equation  dominates (though with a possible shift of the Fujita
exponent), which was the focus of 
\cite{Ikehata-Todorova-Yordanov} and 
\cite{Nishihara-Sobajima-Wakasugi}.  Anticipated $\delta=0$ behavior
is more delicate and results in an $a$-dependent shift of the Strauss
exponent (for small enough $a$).  The authors will explore this latter
case in a sequel to this paper.

In the current study, we focus only on the $\delta>0$ case.  Upper
bounds on the lifespan when $p\le 1+\sqrt{2}$ were established in
\cite{Lai-Liu-Tu-Wang}.  And provided $a\ll 1$, the global existence
result of \cite{MW} would apply.  To the authors' knowledge, however,
no small data global existence result has been proven when
$p>1+\sqrt{2}$ and for any $a, \delta>0$.  The main result of this paper
establishes such and states:
\begin{theorem}\label{main_thm}
  For $p>1+\sqrt{2}$, $a, \delta>0$, $w_{(0)}, w_{(1)}\in C^\infty_c(\R^3)$, and $\varepsilon>0$
  sufficiently small, there exists a global solution $w$ to \eqref{mainEquation}.
\end{theorem}

We note that there are significant additional studies for wave
equations with temporally
dependent damping of the form
\[\frac{a}{(1+t)^{1+\delta}}\partial_t,\quad \delta>0.\]
See \cite{Lai-Takamura} for upper bounds on the lifespan when
$p<p_c(n)$ and \cite{Wakasa-Yordanov} for such upper bounds when
$p=p_c(n)$.  The analog of Theorem \ref{main_thm} was proved
in \cite{Liu-Wang}.  In their case, however, the time-integrability of
the damping coefficient provides a significant advantage that is not
present in our study.

Unlike the recent studies \cite{HMSSZ}, \cite{FaWa}, \cite{LMSTW},
\cite{MW}, our method is not
based on a weighted Strichartz estimate.  We instead only rely upon
integrated local energy estimates.
Due to the globally integrable (in $x$) damping coefficient, the equation can
be conjugated to turn the lower-order damping term $\partial_t$ into
the ``good'' derivative $\partial_t+\partial_r$.  To control this
perturbation, $r^p$-weighted \cite{DaRo_rp} and ghost weighted \cite{Alinhac_ghost} integrated local
energy estimates are used.  The combination of such is as in
\cite{M-Rhoads}.  To finish and close the argument, we use variants of
the space-time Klainerman-Sobolev estimates of \cite{MTT}. 

\vskip .25in
\section{Notation}
We use the common notation $r=|x|$, $\omega_j = \frac{x_j}{r}$, and $\partial_r =
\frac{x}{r}\cdot\nabla$.  This allows us to orthogonally decompose
\[\nabla = \frac{x}{r}\partial_r + \ang.\]
Since we will not have regularity issues at the origin, we will use
$\la r\ra=1+r$ and $\la t-r\ra = 1+t-r$.

We will work with a restricted set of the invariant vector fields.  We
allow
\[S=t\partial_t+r\partial_r,\qquad \Omega = x\times \nabla\]
to denote the generators of scaling and spatial rotations
respectively.  Then, including the generators of (spatial) translations, we set
\[Z=(\nabla, \Omega, S),\]
which is our set of admissible vector fields.

Importantly, these
satisfy
\begin{equation}\label{vfCommutators}[\Box,
  \nabla]=[\Box,\Omega]=0,\quad [\Box, S] = 2\Box.
\end{equation}
Due to the damping, the sign convention will be important, and we are
taking $\Box=\partial_t^2-\Delta$.

We use the shorthand
\[|A Z^{\le N} w|=\sum_{|\mu|\le N} |A Z^\mu w|,\qquad |A \nabla^{\le N}
  w| = \sum_{|\mu|\le N} |A \nabla_x^\mu w|,\qquad |A \partial^{\le N} w|
  = \sum_{|\mu|\le N} |A \partial^\mu w|
\]
for any vector field $A$.
As is implied in this notation, we let $\partial = (\partial_t,
\nabla)$ denote the space-time gradient.
  
The principal decay estimates are variants of those from \cite{MTT}
and are space-time estimates of Klainerman-Sobolev type.  These
bounds take different forms depending on the proximity to the light
cone.  As the initial data are supported in $\{r\le 1/2\}$, by the
finite speed of propagation, the solutions will be supported in
\[C = \{(t,x)\in \R_+\times \R^3\,:\, r\le t+1\}.\]
We first dyadically decompose in time:
\[C_\tau = C\cap \{\la t\ra\in [\tau, 2\tau]\}.\]
In the sequel, the value $\tau$, along with  $R$ and $U$, should be
understood to run over dyadic values.  We further divide into
\[C^{R=1}_\tau = C_\tau \cap \{r\le 2\}, \quad C_\tau^R = C_\tau
  \cap \{r\in [R, 2R]\} \text{ when } R>1,\]
\[C^{U=1}_\tau = C_\tau \cap \{|t-r|\le 2\},\quad C_\tau^U = C_\tau
  \cap \{t-r\in [U, 2U]\} \text{ when } U>1.\]
This allows us to write
\[C_\tau = \bigcup_{R\le \tau/4} C_\tau^R \cup \bigcup_{U\le \tau/4}
  C^U_\tau \cup C^{\tau/2}_\tau,\]
where
\[C^{\tau/2}_\tau = C_\tau \cap \{r\ge \tau/2\}\cap \{t-r\ge \tau/2\}.\]
See, e.g., \cite{M-Stewart} for a graphic representation of this
breakdown.  We use $\tC^R_\tau$ and $\tC^U_\tau$ to represent slight
enlargements of $C^R_\tau$ and $C^U_\tau$ in both parameters.  These
enlargements capture the tails of the cutoffs that are used in the
Sobolev embeddings that follow.

The key properties are that
\[\la r\ra \approx R,\quad t-r\approx t\approx \tau \quad \text{ on } \tC^R_\tau
  \text{ with } 1\le R\le \tau/4\]
and
\[\la t-r\ra \approx U,\quad r\approx t\approx \tau \quad\text{ on } \tC^U_\tau
  \text{ with } 1\le U\le \tau/4.\]
The region $C^{\tau/2}_\tau$ can be informally thought of as either
type as
\[r\approx t\approx t-r\approx \tau\quad\text{ on } \tC^{\tau/2}_\tau.\]
In an abuse of notation, we will typically include it as
$C^{R=\tau/2}_\tau$ in what follows.

To aid with tracking summability in $\tau, R, U$, we introduce the
notation
\begin{align*} \|w\|^2_{\ell^2_\tau\ell^2_{U\le \tau/4} L^2_tL^2_x(C^U_\tau)} &=
  \sum_\tau \sum_{1\le U\le \tau/4}
  \|w\|^2_{L^2_tL^2_x(C^U_\tau)},\\
  \|w\|^2_{\ell^\infty_U \ell^2_{\tau\ge 4U} L^2_tL^2_x(C^U_\tau)} &= \sup_{U\ge
      1} \sum_{\tau\ge 4U} \|w\|^2_{L^2_tL^2_x(C^U_\tau)}
\end{align*}
with variants for $\ell^p$, $L^p$, $C^R_\tau$, etc. defined analogously.

\vskip .25in

\section{Energy and integrated local energy bounds}

In this section we gather the energy estimates on solutions to
the damped wave equation.  In each case, we let $w\in C^2(\R_+\times \R^3)$
solve
\begin{equation}\label{boxw}
  \begin{cases}
    \Box w + \frac{a}{(1+|x|)^{1+\delta}} \partial_t w = F,\quad
    (t,x)\in \R_+\times \R^3,\\
    w(0,\cd)=w_{(0)},\quad \partial_tw(0,\cd)=w_{(1)}
  \end{cases}
\end{equation}
where
\begin{equation}
  \label{supp}
  \supp\, w_{(0)}, w_{(1)}\subseteq \{r\le \frac{1}{2}\},\quad \supp\, F\subseteq
  \{r\le t+\frac{1}{2}\}.
\end{equation}
It is quite likely that the support condition can be relaxed
throughout this paper.

We begin with a standard energy inequality.  Due
to the damping, as a consequence of the usual argument, we pick up a local
energy bound on $\partial_t w$, as is well-known.
\begin{lemma}
For $a>0$, $\delta>0$, and $w\in C^2(\R_+\times \R^3)$ solving \eqref{boxw} subject to
\eqref{supp}, we have
\begin{equation}
  \label{energy}
  \|\partial w\|^2_{L^\infty_t L^2_x} + \|\la r\ra^{-\frac{1+\delta}{2}}
  \partial_t w\|^2_{L^2_{t,x}} \lesssim \|\partial w(0,\cd)\|^2_{L^2}
  + \int_0^\infty \int |F| |\partial_t w|\,dx\,dt.
\end{equation}
\end{lemma}
An elementary argument gives
\[\frac{1}{2} \|\partial w(T,\cd)\|^2_{L^2} = \frac{1}{2}\|\partial
  w(0,\cd)\|^2_{L^2} + \int_0^T\int \Box w \partial_t w\,dx\,dt.\]
Replacing $\Box w = F - a\la r\ra^{-1-\delta}\partial_t w$ and taking
a supremum in $T$ completes the proof.

We subsequently consider a higher order variant of this energy
estimate where vector fields are applied to the solution $w$.
\begin{corollary}
  For $a>0$, $\delta>0$, and $w\in C^3(\R_+\times \R^3)$ solving
  \eqref{boxw} subject to \eqref{supp}, we have
      \begin{equation}
      \label{energyZ1}
      \|\partial Z^{\le 1} w\|^2_{L^\infty_tL^2_x} + \|\la
r\ra^{-\frac{1+\delta}{2}} \partial_t Z^{\le 1} w\|^2_{L^2_{t,x}}
\lesssim \|\partial Z^{\le 1} w(0,\cd)\|^2_{L^2}\\+ \int_0^\infty\int
|Z^{\le 1} F| |\partial_t Z^{\le 1}w|\,dx\,dt.
    \end{equation}
\end{corollary}

\begin{proof}
Letting $L=\Box + \frac{a}{\la r\ra^{1+\delta}} \partial_t$, we
  observe that
  \begin{equation}\label{Lcommutator}[L, \nabla_{x_j}]
    =\omega_j\frac{a(1+\delta)}{\la r\ra^{2+\delta}}\partial_t ,\qquad [L,\Omega]=0,\qquad [L, S]
    = 2L +\frac{\delta a}{\la r\ra^{1+\delta}}\partial_t
    -\frac{a(1+\delta)}{\la r\ra^{2+\delta}}\partial_t.
  \end{equation}
  Applying \eqref{energy} to $Z w$ gives
 \begin{multline*}
\|\partial Z w\|^2_{L^\infty_tL^2_x} + \|\la
r\ra^{-\frac{1+\delta}{2}} \partial_t Z w\|^2_{L^2_{t,x}}
\lesssim \|\partial Z w(0,\cd)\|^2_{L^2}+ \int_0^\infty\int
|Z^{\le 1} F| |\partial_t Zw|\,dx\,dt
\\+ \|\la r\ra^{-\frac{1+\delta}{2}}\partial_t Zw\|_{L^2_{t,x}}\|\la r\ra^{-\frac{1+\delta}{2}}\partial_t w\|_{L^2_{t,x}}.
   \end{multline*}
The first factor of the last term can be absorbed into the left side
by Young's inequality,
   and using \eqref{energy} to bound the second factor of the last term gives \eqref{energyZ1}
\end{proof}

The next estimate to explore is the now commonly applied integrated local
energy estimate.  This type of estimate originates from \cite{Mo2}
with the proof that we outline from \cite{Sterbenz}, \cite{MetSo06}.
\begin{lemma}
  For $a>0$, $\delta>0$, and $w\in C^2(\R_+\times \R^3)$ solving \eqref{boxw} subject to
  \eqref{supp}, we have
  \begin{equation}
    \label{le}
   \|\la r\ra^{-\frac{1}{2}}\partial w\|_{\ell^\infty_R \ell^2_\tau
     L^2_{t,x}(C^R_\tau)}
+ \|\la r\ra^{-\frac{3}{2}} w\|_{\ell^\infty_R \ell^2_\tau L^2_{t,x}(C^R_\tau)}
   \lesssim \|\partial w(0,\cd)\|_{L^2}
   + \|\la r\ra^{\frac{1}{2}} F\|_{\ell^1_R \ell^2_\tau L^2_{t,x}(C^R_\tau)}.
  \end{equation}
\end{lemma}

When $R\approx \tau$, we can further subdivide $C^R_\tau$ dyadically in
$t-|x|$, which allows the naive estimate
\begin{equation}\label{Fterm}\|\la r\ra^{\frac{1}{2}} F\|_{\ell^1_R \ell^2_\tau
    L^2_{t,x}(C^R_\tau)}
  \lesssim \inf_{F=F_1+F_2} \Bigl(\|\la r\ra ^{\frac{3}{2}-s} F_1\|_{\ell^2_\tau \ell^2_{R} L^2_{t,x}(C^R_\tau)} + \|\la r\ra^{1-s} \la
  t-r\ra^{\frac{1}{2}} F_2\|_{\ell^1_{U} \ell^2_{\tau}
    L^2_{t,x}(C^U_\tau)}\Bigr)
\end{equation}
provided that $s\le \frac{1}{2}$.
We will most commonly apply this with
\begin{equation}\label{Fsplit} F_1=\bigcup_{\tau}\bigcup_{1\le R\le \tau/2} \chi_{C^R_\tau} F,\qquad
  F_2 = \bigcup_\tau \bigcup_{1\le U\le \tau/4} \chi_{C^U_\tau} F.\end{equation}
The norms on the right of \eqref{Fterm} will
appear in later bounds, and it will be convenient to have them in this
common form.

\begin{proof}
For $R>0$ fixed, the argument of \cite{Sterbenz}, \cite{MetSo06} yields 
\begin{multline*}
  \int_0^T\int_{\{\la x\ra\approx R\}} \la r\ra^{-1} |\partial
  w|^2\,dx\,dt + \int_0^T\int_{\{\la x\ra\approx R\}} \la r\ra^{-3} |w|^2\,dx\,dt
 \\ \lesssim \int |\partial w(0,x)|^2\,dx + \int_0^T\int |\Box w|
  \Bigl(|\partial w| +\la r\ra^{-1} |w|\Bigr)\,dx\,dt.
\end{multline*}
To see this, one pairs $\Box w$ with a multiplier of the form
$C\partial_t + \frac{r}{r+R}\partial_r + \frac{1}{r+R}$, integrates by
parts, and applies the Hardy inequality
\begin{equation}\label{L2Hardy}
  \|r^{-1}f\|_{L^2(\R^3)}\lesssim \|\nabla f\|_{L^2(\R^3)}.
\end{equation}
Taking a supremum in $T, R$ and
applying the Schwarz inequality yields
\begin{multline*}
  \|\la r\ra^{-\frac{1}{2}} \partial w\|^2_{\ell^\infty_R \ell^2_\tau
    L^2_{t,x}(C^R_\tau)} +
    \|\la r\ra^{-\frac{3}{2}} w\|^2_{\ell^\infty_R \ell^2_\tau
      L^2_{t,x}(C^R_\tau)}
    \lesssim \|\partial w(0,\cd)\|^2_{L^2} \\+ \|\la r\ra^{\frac{1}{2}}
    \Box w\|_{\ell^1_R\ell^2_\tau L^2_{t,x}(C^R_\tau)}
    \Bigl(\|\la r\ra^{-\frac{1}{2}} \partial w\|_{\ell^\infty_R \ell^2_\tau
    L^2_{t,x}(C^R_\tau)} +
    \|\la r\ra^{-\frac{3}{2}} w\|_{\ell^\infty_R \ell^2_\tau
      L^2_{t,x}(C^R_\tau)}\Bigr).
\end{multline*}
Using Young's inequality, we may absorb the second factor of the last
term into the left side, which gives
\begin{multline*}
   \|\la r\ra^{-\frac{1}{2}} \partial w\|_{\ell^\infty_R \ell^2_\tau
    L^2_{t,x}(C^R_\tau)} +
    \|\la r\ra^{-\frac{3}{2}} w\|_{\ell^\infty_R \ell^2_\tau
      L^2_{t,x}(C^R_\tau)}
    \lesssim \|\partial w(0,\cd)\|_{L^2} \\+ \|\la r\ra^{\frac{1}{2}}
    F\|_{\ell^1_R\ell^2_\tau L^2_{t,x}(C^R_\tau)}
    + \|\la r\ra^{-\frac{1}{2}-\delta} \partial_t
    w\|_{\ell^1_R\ell^2_\tau L^2_{t,x}(C^R_\tau)}.
\end{multline*}
Applying \eqref{energy}, the last term is controlled via
\begin{align*}
\|\la r\ra^{-\frac{1}{2}-\delta} \partial_t
    w\|_{\ell^1_R\ell^2_\tau L^2_{t,x}(C^R_\tau)}^2
&\lesssim \|\la r\ra^{-\frac{1+\delta}{2}} \partial_t
w\|^2_{L^2_{t,x}}\\
&\lesssim \|\partial w(0,\cd)\|_{L^2} + \int_0^\infty \int |F|
|\partial w|\,dx\,dt.
\end{align*}
The Schwarz inequality on the last term lets us see that
\[\int_0^\infty \int |F| |\partial w|\,dx\,dt
  \lesssim \|\la r\ra^{\frac{1}{2}}F\|_{\ell^1_R\ell^2_\tau L^2_{t,x}(C^R_\tau)}
  \|\la r\ra^{-\frac{1}{2}} \partial w\|_{\ell^\infty_R
    \ell^2_\tau L^2_{t,x}(C^R_\tau)},\]
and we may again absorb the last factor into the far left side to obtain \eqref{le}.
\end{proof}

A higher order version of the previous estimate, where up to two vector fields are applied to
the solution, will be required and is the next result.
\begin{corollary}
For $a>0$, $\delta>0$, and $w\in C^4(\R_+\times \R^3)$ solving
\eqref{boxw} subject to \eqref{supp}, we have
\begin{multline}
  \label{leZ2}
   \|\la r\ra^{-\frac{1}{2}}\partial Z^{\le 2} w\|_{\ell^\infty_R \ell^2_\tau
     L^2_{t,x}(C^R_\tau)}
+ \|\la r\ra^{-\frac{3}{2}} Z^{\le 2} w\|_{\ell^\infty_R \ell^2_\tau L^2_{t,x}(C^R_\tau)}
   \\\lesssim \|\partial Z^{\le 2} w(0,\cd)\|_{L^2}
   + \|\la r\ra^{\frac{1}{2}} Z^{\le 2} F\|_{\ell^1_R \ell^2_\tau L^2_{t,x}(C^R_\tau)}.
\end{multline}
\end{corollary}

\begin{proof}
Applying \eqref{le} to $Z^{\le 2} w$ and using \eqref{Lcommutator}, we have
  \begin{multline}\label{leZ2_1}\|\la r\ra^{-\frac{1}{2}}\partial Z^{\le 2} w\|_{\ell^\infty_R
      \ell^2_\tau L^2_{t,x}(C^R_\tau)} + \|\la r\ra^{-\frac{3}{2}}
    Z^{\le 2} w\|_{\ell^\infty_R\ell^2_\tau L^2_{t,x}(C^R_\tau)}
    \lesssim \|\partial Z^{\le 2} w(0,\cd)\|_{L^2}
    \\+ \|\la r\ra^{\frac{1}{2}} Z^{\le 2} F\|_{\ell^1_R\ell^2_\tau
      L^2_{t,x}(C^R_\tau)}
    +\|\la r\ra^{-\frac{1}{2}-\delta} \partial_t Z^{\le 1}
    w\|_{\ell^1_R\ell^2_\tau L^2_{t,x}(C^R_\tau)}
+ \|r^{-1} \la r\ra^{-\frac{3}{2}-\delta} \partial_t
w\|_{\ell^1_R\ell^2_\tau L^2_{t,x}(C^R_\tau)}.
    \end{multline}
    For the second to last term, we have
    \[\|\la r\ra^{-\frac{1}{2}-\delta} \partial_t Z^{\le 1}
      w\|_{\ell^1_R\ell^2_\tau L^2_{t,x}(C^R_\tau)}
      \lesssim \|\la r\ra^{-\frac{1+\delta}{2}}\partial_t Z^{\le 1}
      w\|_{L^2_{t,x}},\]
For the last term in \eqref{leZ2_1}, which results in the case that
both vector fields are $\nabla$ and the second one lands on the
$\omega_j$ in \eqref{Lcommutator}, we apply the Hardy estimate
\eqref{L2Hardy}
to get
\begin{align*}
  \|r^{-1} \la r\ra^{-\frac{3}{2}-\delta} \partial_t
  w\|_{\ell^1_R\ell^2_\tau L^2_{t,x}(C^R_\tau)} &\lesssim
                                                    \|\la
                                                  r\ra^{-\frac{3}{2}-\delta}
                                                  \partial_t \nabla
                                                  w\|_{\ell^1_R\ell^2_\tau L^2_{t,x}(C^R_\tau)}
                                                  +  \|\la r\ra^{-\frac{5}{2}-\delta} \partial_t
                                                  w\|_{\ell^1_R\ell^2_\tau L^2_{t,x}(C^R_\tau)}\\
  &\lesssim \|\la r\ra^{-\frac{1+\delta}{2}} \partial_t Z^{\le 1} w\|_{L^2_{t,x}}.
\end{align*}
In both cases we may apply \eqref{energyZ1}.
   Using the Schwarz
    inequality as above to see that
    \[\int_0^\infty \int |Z^{\le 1}F||\partial_tZ^{\le 1}w|\,dx\,dt
      \lesssim \|\la r\ra^{\frac{1}{2}} Z^{\le
        1}F\|_{\ell^1_R\ell^2_\tau L^2_{t,x}(C^R_\tau)}
      \|\la r\ra^{-\frac{1}{2}} \partial_t Z^{\le 1}
      w\|_{\ell^\infty_R\ell^2_\tau L^2_{t,x}(C^R_\tau)}\]
    and using Young's inequality to absorb the second factor back into
    the left side of \eqref{leZ2_1}.
  \end{proof}


The remaining results in this section will rely on a conjugation.  The
conjugation turns the first order term into the ``good'' derivative
$\partial_t+\partial_r$, which meshes nicely with the $r^p$/ghost
weighted estimates in the sequel.  The remaining error terms either
have extra decay at infinity, which yields smallness, or are supported
in a compact set where the weights are irrelevant due to \eqref{le}.

In what follows, we fix $\beta$ to be a smooth, radial function that
vanishes for $\{r\le 1\}$ and is identically $1$ on $\{r\ge 2\}$.
For $w$ solving \eqref{boxw}, we set
\begin{equation}
  \label{v}
  v = \exp\Bigl(-\frac{a}{2\delta} \la r\ra^{-\delta}\beta(r)\Bigr)
  w.
\end{equation}
Then $v$ solves
\begin{multline}
  \label{vequation}
  \Bigl[\Box +
  a\la r\ra^{-1-\delta}\beta(r)\Bigl(\partial_t+\partial_r+\frac{1}{r}\Bigr)\Bigr]v
  = e^{-\frac{a}{2\delta}\la r\ra^{-\delta}\beta(r)}F
+ \frac{a(1+\delta)}{2}
\beta(r) \la r\ra^{-2-\delta}v \\+
\frac{a^2}{4}  \la r\ra^{-2-2\delta} (\beta(r))^2 v
+\frac{a}{\delta} \la r\ra^{-\delta}\beta'(r) \partial_r v
-a\la r\ra^{-1-\delta} (1-\beta(r)) \partial_t v
-V(r) v
\end{multline}
where
\[V(r)=a\la r\ra^{-1-\delta} \beta'(r) -
  \frac{a}{2\delta}\la r\ra^{-\delta}\Delta \beta(r) +
  \frac{a^2}{2\delta}\la r\ra^{-1-2\delta}\beta(r) \beta'(r)
  -\frac{a^2}{4\delta^2}\la r\ra^{-2\delta}(\beta'(r))^2,\]
which importantly is smooth and has support contained within $\{1\le r\le 2\}$.  We include
the cutoff $\beta$ to eliminate singular behaviors that might
otherwise occur when commuting the generators of translations,
$\nabla_{x_j}$, with the linear operator.  We will use the following
abbreviation for the linear operator
\begin{multline}
  \label{L}
  \L = \Box + \frac{a}{\la
    r\ra^{1+\delta}}\beta(r)\Bigl(\partial_t+\partial_r+\frac{1}{r}\Bigr)
  - \frac{a(1+\delta)}{2}
\beta(r) \la r\ra^{-2-\delta} \\-
\frac{a^2}{4}  \la r\ra^{-2-2\delta} (\beta(r))^2 
-\frac{a}{\delta} \la r\ra^{-\delta}\beta'(r)\partial_r 
+a\la r\ra^{-1-\delta} (1-\beta(r)) \partial_t 
+V(r),
\end{multline}
and as such we have $\L v = e^{-\frac{a}{2\delta}\la
  r\ra^{-\delta}\beta(r)} F$.

The next theorem is akin to \cite[Proposition 2.3]{M-Stewart} applied
to $v$.  This represents a combination of the $r^p$-weighted estimates
of \cite{DaRo_rp} and the ghost weighted estimates of
\cite{Alinhac_ghost} as was done in \cite{M-Rhoads}.  We work with
$\frac{p}{2}=1-s$, where in this case the $p$ represents that within
the $r^p$-estimates as in \cite[Proposition 2.3]{M-Stewart}.
\begin{theorem}\label{thm_rp}
  Let $0<s\le 1/2$ and $0\le \sigma< s$.  Then if $w\in C^2(\R_+\times \R^3)$ solves
  \eqref{boxw} subject to \eqref{supp}, we have
  \begin{multline}\label{rp}
   \|\la r\ra^{\frac{1}{2}-s} r^{-1}
w\|_{L^2_{t,x}}+
\|\la r\ra^{\frac{1}{2}-s}
r^{-1}(\partial_t+\partial_r)(r w)\|_{L^2_{t,x}}
+\|\la r\ra^{\frac{1}{2}-s} \ang w\|_{L^2_{t,x}}
  \\+ \|\la r\ra^{\frac{1}{2}-s+\sigma} \la t-r\ra^{-\sigma}
  r^{-1}(\partial_t+\partial_r)(r
  w)\|_{\ell^\frac{2}{1-2\sigma}_U\ell^2_{\tau\ge 4U}
    L^2_{t,x}(C^U_\tau)}
 \\   + \|\la r\ra^{-\frac{1}{2}-s+\sigma} \la t-r\ra^{-\sigma}
  w\|_{\ell^\frac{2}{1-2\sigma}_U\ell^2_{\tau\ge 4U}
    L^2_{t,x}(C^U_\tau)}
  \\
  \lesssim \|\partial
  w(0,\cd)\|_{L^2}
  +\inf_{F=F_1+F_2}\Bigl(
 \|\la r\ra^{\frac{3}{2}-s} F_1\|_{\ell^2_\tau \ell^2_R L^2_{t,x}(C^R_\tau)}
 +\|\la r\ra^{1-s}\la t-r\ra^{\frac{1}{2}} F_2\|_{\ell^1_U
   \ell^2_\tau L^2_{t,x}(C^U_\tau)}\Bigr).
\end{multline}
\end{theorem}

For the most part, it would suffice to have $0<s<1$ and $0\le \sigma <
\min(s, 1/2)$ in the hypotheses.  The stronger restriction arises
from an application of \eqref{Fterm} in order to measure the forcing
terms from the standard local energy estimate \eqref{le} in the same
spaces that are natural here.  The use of \eqref{le} is convenient to control
some of the error terms in \eqref{vequation}.  In terms of the
original nonlinear problem, the restriction $s\le 1/2$ will correspond to
$p\le 3$.  In the case of the free wave equation, $p=3$ is the conformal
exponent and global existence for such nonlinear equations is known to
follow from Strichartz estimates for powers exceeding this.  The techniques of, e.g.,
\cite{HMSSZ}, \cite{LMSTW}, and \cite{MW} encountered a similar
restriction.  We will discuss the straightforward remedy to obtain
global existence for $p > 3$ in the last section.

To complete the proof of \eqref{rp}, the following Hardy
inequality \cite[Lemma 2.2]{M-Stewart} will be applied.
\begin{lemma}
  Let $0<s<1$, and let $v\in C^2(\R_+\times\R^3)$ be so that for every $T$
  there exist $R$ so that $v(t,x)\equiv 0$ for all $t\in [0,T]$ and
$|x|\ge R$.  Then
  \begin{equation}
    \label{hardyR}
    \|\la r\ra^{\frac{1}{2}-s} r^{-1} v\|_{L^2_{t,x}} + \|\la r\ra^{\frac{1}{2}-s} r^{-\frac{1}{2}}v\|_{L^\infty_tL^2_x} \lesssim
    \|\la r\ra^{\frac{1}{2}-s} r^{-\frac{1}{2}}v(0,\cd)\|_{L^2_x}
    + \|\la r\ra^{\frac{1}{2}-s}
    r^{-1}(\partial_t+\partial_r)(r v )\|_{L^2_{t,x}}.
  \end{equation}
\end{lemma}
This Hardy-type inequality follows by noticing
\[\int_0^T\int_{\R^3} \frac{\la r\ra^{1-2s}}{r^2} v^2\,dx\,dt
  = - \int_0^T\int_{\S^2}\int_0^\infty
  \la r\ra^{1-2s}(\partial_t+\partial_r)(r^{-1})
  (rv)^2\,dr\,d\omega\,dt\]
and integrating by parts.

To obtain the main $L^2_{t,x}$-estimate on the undifferentiated
solution, another Hardy inequality is needed.  One, perhaps, natural
thing to consider would be an estimate like
\[\|\la r\ra^{-s} \la t-r\ra^{-\frac{1}{2}} v\|_{\ell^\infty_U \ell^2_\tau
    L^2_{t,x}(C^U_\tau)}
  \lesssim \|v(0,\cd)\|_{L^2} + \|\la r\ra^{1-s} \la t-r\ra^{-\frac{1}{2}}
  r^{-1}(\partial_t+\partial_r)(rv)\|_{\ell^\infty_U \ell^2_\tau
    L^2_{t,x} (C^U_\tau)},\]
which is akin to \cite[Lemma 2.4]{M-Rhoads}.  But the proof of this estimate
requires that $s>1/2$, which is an insufficiently broad range for the
allowable $p$ to reach
down to $p_c(3)$.

As it will be necessary independently, we first work with a good
derivative term.  We assume that the support of $v$ is contained in
$\{r\le t+\frac{1}{2}\}$.  By H\"older's inequality, for each fixed $U$, we have
\begin{multline}\label{goodUsum}
 \Bigl\|\la r\ra^{\frac{1}{2}-s+\sigma} \la t-r\ra^{-\sigma}
 r^{-1}(\partial_t+\partial_r)(r v)\Bigr\|_{\ell^2_\tau L^2_{t,x}(C^U_\tau)}
  \\\lesssim  \Bigl\|\la r\ra^{\frac{1}{2}-s}
  r^{-1}(\partial_t+\partial_r)(rv)\Bigr\|_{\ell^2_\tau L^2_{t,x}(C^U_\tau)}^{1-2\sigma}
  \Bigl\|\la r\ra^{1-s} \la t-r\ra^{-\frac{1}{2}}
  r^{-1}(\partial_t+\partial_r)(rv)\Bigr\|_{\ell^2_\tau L^2_{t,x}(C^U_\tau)}^{2\sigma}
\end{multline}
provided that $0\le \sigma\le \frac{1}{2}$.  Thus,
\begin{multline}\label{goodUsum2}
  \Bigl\|\la r\ra^{\frac{1}{2}-s+\sigma} \la t-r\ra^{-\sigma}
 r^{-1}(\partial_t+\partial_r)(rv)\Bigr\|_{\ell^{\frac{2}{1-2\sigma}}_U\ell^2_\tau L^2_{t,x}(C^U_\tau)}
 \\\lesssim 
  \Bigl\|\la r\ra^{\frac{1}{2}-s}
  r^{-1}(\partial_t+\partial_r)(rv)\Bigr\|_{L^2_{t,x}}^{1-2\sigma}
  \Bigl\|\la r\ra^{1-s} \la t-r\ra^{-\frac{1}{2}}
  r^{-1}(\partial_t+\partial_r)(rv)\Bigr\|_{\ell^\infty_U \ell^2_\tau L^2_{t,x}(C^U_\tau)}^{2\sigma}.
\end{multline}

The $C^U_\tau$-based Hardy inequality that we shall use is:
\begin{lemma}\label{lem_hardyU}
  Let $0\le \sigma < \min(s,1/2)$, and let $v\in C^2(\R_+\times\R^3)$
  have support that is contained in $\{r\le t+\frac{1}{2}\}$.
Then
  \begin{multline}
    \label{hardyU}
  \Bigl\|\la r\ra^{-\frac{1}{2}-s+\sigma} \la t-r\ra^{-\sigma}
  v\Bigr\|_{\ell^{\frac{2}{1-2\sigma}}_U \ell^2_\tau L^2_{t,x}(C^U_\tau)}
  \lesssim  \|v(0,\cd)\|_{L^2}
\\  + \Bigl\|\la r\ra^{\frac{1}{2}-s}
  r^{-1}(\partial_t+\partial_r)(rv)\Bigr\|_{L^2_{t,x}}^{1-2\sigma}
  \Bigl\|\la r\ra^{1-s} \la t-r\ra^{-\frac{1}{2}}
  r^{-1}(\partial_t+\partial_r)(rv)\Bigr\|_{\ell^\infty_U \ell^2_\tau L^2_{t,x}(C^U_\tau)}^{2\sigma}.
  \end{multline}
\end{lemma}

\begin{proof}
  We bound 
  \begin{multline*}
\Bigl\|\la r\ra^{-\frac{1}{2}-s+\sigma} \la t-r\ra^{-\sigma}
  v\Bigr\|^2_{\ell^2_{\tau\le T} L^2_{t,x}(C^U_\tau)}\\\le
    \frac{1}{-2s+2\sigma}\int_0^T\int_{\S^2}\int_0^\infty
    (\partial_t+\partial_r)\la r\ra^{-2s+2\sigma}
    \la t-r\ra^{-2\sigma} \chi_U(t-r)
    (rv)^2\,dr\,d\omega\,dt.
  \end{multline*}
  Here $\chi_U$ is a characteristic function that localizes to
  $(1+t-r)\approx U$.
  Integration by parts shows that this is equivalent to
  \begin{multline*}
    \frac{1}{-2s+2\sigma} \int_{\R^3}\la r\ra^{-2s+2\sigma} \la
    t-r\ra^{-2\sigma} \chi_U(t-r) v^2\,dx\Bigl|_{t=0}^T
    \\+\frac{1}{s-\sigma} \int_0^T\int_{\S^2}\int_0^\infty
    \la r\ra^{-2s+2\sigma} \la t-r\ra^{-2\sigma}
    \chi_U(t-r) (rv)\cdot (\partial_t+\partial_r)(rv)\,dr\,d\omega\,dt.
  \end{multline*}
  Applying the Schwarz inequality to the last term, we bound it by
  \begin{multline*}
   \Bigl(\int_0^T\int_{\S^2}\int_0^\infty
    \la r\ra^{-1-2s+2\sigma}
    \la t-r\ra^{-2\sigma} \chi_U(t-r)
    (rv)^2\,dr\,d\omega\,dt\Bigr)^{1/2}
    \\\times
     \Bigl(\int_0^T\int_{\S^2}\int_0^\infty
    \la r\ra^{1-2s+2\sigma}
    \la t-r\ra^{-2\sigma} \chi_U(t-r)
    [(\partial_t+\partial_r)(rv)]^2\,dr\,d\omega\,dt\Bigr)^{1/2}.
  \end{multline*}
Using Young's inequality, the first factor can be absorbed back into
the left side.  Dropping the beneficially signed term on the $t=T$
slice, taking a supremum in $T$, and applying \eqref{goodUsum}
gives
\begin{multline*}
  \Bigl\|\la r\ra^{-\frac{1}{2}-s+\sigma}\la t-r\ra^{-\sigma}
  v\Bigr\|_{\ell^2_\tau L^2_{t,x}(C^U_\tau)}
  \lesssim \delta_{1U} \|v(0,\cd)\|_{L^2}
\\  + \Bigl\|\la r\ra^{\frac{1}{2}-s}
  r^{-1}(\partial_t+\partial_r)(rv)\Bigr\|_{\ell^2_\tau L^2_{t,x}(C^U_\tau)}^{1-2\sigma}
  \Bigl\|\la r\ra^{1-s} \la t-r\ra^{-\frac{1}{2}}
  r^{-1}(\partial_t+\partial_r)(rw)\Bigr\|_{\ell^2_\tau L^2_{t,x}(C^U_\tau)}^{2\sigma}.
\end{multline*}
We may then sum in $U$ and apply H\"older's inequality to complete the proof.
\end{proof}

\begin{proof}[Proof of Theorem \ref{thm_rp}]
  With $\sigma_\theta(z)=\frac{z}{\theta+|z|}$ and $v$ as in \eqref{v}, we consider
  \[\int_0^T\int_{\R^3} \la r\ra^{2-2s}e^{-\sigma_U(t-|x|)} \Bigl[\Box
    +
    a\la r\ra^{-1-\delta}\beta(r)\Bigl(\partial_t+\partial_r+\frac{1}{r}\Bigr)\Bigr]v\Bigl(\partial_t+\partial_r+\frac{1}{r}\Bigr)v\,dx\,dt.\]
  Using that
  \[\Box v =
    r^{-1}\Bigl(\partial_t^2-\partial_r^2-\ang\cdot\ang\Bigr)(rv),\quad
    \Bigl(\partial_t+\partial_r+\frac{1}{r}\Bigr)v =
    r^{-1}(\partial_t+\partial_r)(rv),\]
  this becomes
  \begin{multline}
    \label{rp1}
   \frac{1}{2}\int_0^T\int_{\S^2} \int_0^\infty \la r\ra^{2-2s}
     e^{-\sigma_U(t-r)}
     (\partial_t-\partial_r)\Bigl[(\partial_t+\partial_r)(rv)\Bigr]^2\,dr\,d\omega\,dt
    \\+ \int_0^T\int_{\S^2} \int_0^\infty
    \la r\ra^{2-2s}e^{-\sigma_U(t-r)} \frac{1}{r}
    |\ang(rv)|^2\,dr\,d\omega\,dt
    \\+ \frac{1}{2} \int_0^T\int_{\S^2}\int_0^\infty \la r\ra^{2-2s}
    e^{-\sigma_U(t-r)}
    (\partial_t+\partial_r)|\ang(rv)|^2\,dr\,d\omega\,dt
   \\ + a\int_0^T\int_{\S^2}\int_0^\infty \la r\ra^{1-2s-\delta} \beta(r)e^{-\sigma_U(t-r)}\Bigl[(\partial_t+\partial_r)(rv)\Bigr]^2\,dr\,d\omega\,dt.
  \end{multline}
  Here we have used the fact that
  \begin{equation}\label{angcomm}[\ang, \partial_r] = \frac{1}{r}\ang.\end{equation}
  Integrating by parts then shows that \eqref{rp1} is equivalent to
  \begin{multline}
    \label{rp2}
       \frac{1}{2}\int_{\S^2} \int_0^\infty \la r\ra^{2-2s}
     e^{-\sigma_U(t-r)} \Bigl\{
     \Bigl[(\partial_t+\partial_r)(rv)\Bigr]^2+|\ang(rv)|^2\Bigr\}\,dr\,d\omega
     \Bigl|_{t=0}^T
     + \frac{1}{2}\int_0^T\int_{\S^2} e^{-\sigma_U(t)} v^2(t,0)\,d\omega\,dt
     \\+
       (1-s)\int_0^T\int_{\S^2} \int_0^\infty \la r\ra^{1-2s}
     e^{-\sigma_U(t-r)}
     \Bigl[(\partial_t+\partial_r)(rv)\Bigr]^2\,dr\,d\omega\,dt
     \\+
\int_0^T\int_{\S^2} \int_0^\infty \la r\ra^{2-2s} \sigma_U'(t-r)
     e^{-\sigma_U(t-r)}
  \Bigl[(\partial_t+\partial_r)(rv)\Bigr]^2\,dr\,d\omega\,dt
    \\+ \int_0^T\int_{\S^2} \int_0^\infty
    \la r\ra^{2-2s}e^{-\sigma_U(t-r)} \frac{1}{r}
    |\ang(rv)|^2\,dr\,d\omega\,dt
    \\- (1-s) \int_0^T\int_{\S^2}\int_0^\infty \la r\ra^{1-2s}
    e^{-\sigma_U(t-r)}
|\ang(rv)|^2\,dr\,d\omega\,dt
   \\ + a\int_0^T\int_{\S^2}\int_0^\infty \la r\ra^{1-2s-\delta}\beta(r) e^{-\sigma_U(t-r)}\Bigl[(\partial_t+\partial_r)(rv)\Bigr]^2\,dr\,d\omega\,dt.
  \end{multline}
The beneficially signed second and seventh terms may be dropped in
favor of an inequality.  Provided that $s>0$, the sixth term may be
absorbed into the fifth.  And for the fourth term, we use that
\[\sigma_U'(t-r)\gtrsim U^{-1},\quad \text{ when } \la t-r\ra \approx
  U.\]
Thus, provided that we additionally have $s<1$, upon taking a supremum
in $U, T$ we obtain
\begin{multline}
  \label{rp3}
  \|\la r\ra^{1-s} r^{-1}(\partial_t+\partial_r)(r
  v)\|^2_{L^\infty_tL^2_x} +
    \|\la r\ra^{1-s} \ang v\|^2_{L^\infty_tL^2_x}
+\|\la r\ra^{\frac{1}{2}-s}
  r^{-1}(\partial_t+\partial_r)(rv)\|^2_{L^2_{t,x}}\\+ \|\la r\ra^{\frac{1}{2}-s} |\ang v|\|^2_{L^2_{t,x}} 
  + \|\la r\ra^{1-s} \la t-r\ra^{-\frac{1}{2}}
  r^{-1}(\partial_t+\partial_r)(r v)\|^2_{\ell^\infty_U\ell^2_\tau
    L^2_{t,x}(C^U_\tau)}
  \\
  \lesssim \|\la r\ra^{1-s} r^{-1}(\partial_t+\partial_r)(r
  v)(0,\cd)\|^2_{L^2} + \|\la r\ra^{1-s} \ang v(0,\cd)\|^2_{L^2} \\+
  \int_0^\infty \int_{\R^3}
\la r\ra^{2-2s}\Bigl|\Bigl[\Box
    +
    a\la r\ra^{-1-\delta}\beta(r)\Bigl(\partial_t+\partial_r+\frac{1}{r}\Bigr)\Bigr]v\Bigr|
    \Bigl| r^{-1}(\partial_t+\partial_r)(rv)\Bigr|\,dx\,dt.  
\end{multline}
Here
we are also relying on the boundedness of $\sigma_U(t-r)$.

If $F=F_1+F_2$, substituting \eqref{vequation} and applying the Schwarz
  inequality to the last term in \eqref{rp3} gives
\begin{multline}
    \label{rp4}
  \int_0^\infty \int_{\R^3}
\la r\ra^{2-2s}\Bigl|\Bigl[\Box
    +
    a\la r\ra^{-1-\delta}\beta(r)\Bigl(\partial_t+\partial_r+\frac{1}{r}\Bigr)\Bigr]v\Bigr|
    \Bigl| r^{-1}(\partial_t+\partial_r)(rv)\Bigr|\,dx\,dt
    \\
  \lesssim\|\la r\ra^{\frac{3}{2}-s} F_1\|_{\ell^2_\tau \ell^2_{R} L^2_{t,x}(C^R_\tau)}
 \|\la r\ra^{\frac{1}{2}-s} r^{-1}(\partial_t+\partial_r)(rv)\|_{L^2_{t,x}}
 \\  +\|\la r\ra^{1-s}\la t-r\ra^{\frac{1}{2}} F_2\|_{\ell^1_U \ell^2_{\tau} L^2_{t,x}(C^U_\tau)}
 \|\la r\ra^{1-s}
 \la t-r\ra^{-\frac{1}{2}}r^{-1}(\partial_t+\partial_r)(rv)\|_{\ell^\infty_U
   \ell^2_{\tau} L^2_{t,x}(C^U_\tau)}
 \\+ \|\la r\ra^{-\frac{1}{2}-s-\delta}
 v\|_{L^2_{t,x}}
\|\la r\ra^{\frac{1}{2}-s}
r^{-1}(\partial_t+\partial_r)(rv)||_{L^2_{t,x}}
 \\+ \|\partial^{\le 1} v\|_{L^2_{t,x}(\R_+\times\{r\le 2\})}
\|\la r\ra^{\frac{1}{2}-s} r^{-1}(\partial_t+\partial_r)(rv)||_{L^2_{t,x}}.
\end{multline}
Here we have also used the boundedness of $\frac{a}{2\delta}\la
  r\ra^{-\delta}\beta(r)$.

In each term, we may use Young's inequality to absorb the second
factor into the left side of \eqref{rp3}.  This results in
\begin{multline}
  \label{rp5}
  \|\la r\ra^{\frac{1}{2}-s} r^{-1}
v\|_{L^2_{t,x}}+
\|\la r\ra^{\frac{1}{2}-s}
  r^{-1}(\partial_t+\partial_r)(rv)\|_{L^2_{t,x}} +
  \|\la r\ra^{\frac{1}{2}-s} \ang v\|_{L^2_{t,x}}
  \\+ \|\la r\ra^{1-s} \la t-r\ra^{-\frac{1}{2}}
  r^{-1}(\partial_t+\partial_r)(r v)\|_{\ell^\infty_U\ell^2_\tau
    L^2_{t,x}(C^U_\tau)} 
  \lesssim \|\la r\ra^{1-s} (\partial_t+\partial_r, \ang, r^{-1})
  v(0,\cd)\|_{L^2}
  \\+\inf_{F=F_1+F_2}\Bigl(
 \|\la r\ra^{\frac{3}{2}-s}  F_1\|_{\ell^2_\tau \ell^2_{R} L^2_{t,x}(C^R_\tau)}
 +\|\la r\ra^{1-s}\la t-r\ra^{\frac{1}{2}} F_2\|_{\ell^1_U \ell^2_{\tau} L^2_{t,x}(C^U_\tau)}\Bigr)
 \\+ \|\la r\ra^{-\frac{1}{2}-s-\delta}
v\|_{L^2_{t,x}}+\|\partial^{\le 1} v\|_{L^2_{t,x}(\R_+\times\{r\le 2\})}.
\end{multline}
To simplify
  the exposition, we have dropped the $L^\infty_t$-terms in the left side of
  \eqref{rp3} since they will not be necessary in the 
  sequel.  The first term is, however, a new addition, which is
  available due to \eqref{hardyR}.

Additionally including \eqref{goodUsum2} and \eqref{hardyU}, we get
\begin{multline}
  \label{rp6}
  \|\la r\ra^{\frac{1}{2}-s} r^{-1}
v\|_{L^2_{t,x}}+
\|\la r\ra^{\frac{1}{2}-s}
  r^{-1}(\partial_t+\partial_r)(rv)\|_{L^2_{t,x}} +
  \|\la r\ra^{\frac{1}{2}-s}\ang v\|_{L^2_{t,x}}
  \\+ \|\la r\ra^{\frac{1}{2}-s+\sigma} \la t-r\ra^{-\sigma}
  r^{-1}(\partial_t+\partial_r)(r v)\|_{\ell^\frac{2}{1-2\sigma}_U\ell^2_\tau
    L^2_{t,x}(C^U_\tau)}
   + \|\la r\ra^{-\frac{1}{2}-s+\sigma} \la t-r\ra^{-\sigma}
  v\|_{\ell^\frac{2}{1-2\sigma}_U\ell^2_\tau
    L^2_{t,x}(C^U_\tau)}
  \\
  \lesssim \|\partial
  v(0,\cd)\|_{L^2}
  +\inf_{F=F_1+F_2}\Bigl(
 \|\la r\ra^{\frac{3}{2}-s} F_1\|_{\ell^2_\tau \ell^2_{R} L^2_{t,x}(C^R_\tau)}
 +\|\la r\ra^{1-s}\la t-r\ra^{\frac{1}{2}} F_2\|_{\ell^1_U \ell^2_{\tau} L^2_{t,x}(C^U_\tau)}\Bigr)
 \\+ \|\la r\ra^{-\frac{1}{2}-s-\delta}
v\|_{L^2_{t,x}} +\|\partial^{\le 1} v\|_{L^2_{t,x}(\R_+\times\{r\le 2\})}.
\end{multline}
Here we have also simplified the norm on the initial data using
\eqref{supp} and a standard Hardy inequality \eqref{L2Hardy}.  Though we will need to
repeat a portion of this step below, we take the time to control the
error terms before passing back to $w$.  Fixing $\tilde{R}$, we may bound
\begin{multline*}
\|\la r\ra^{-\frac{1}{2}-s-\delta}
v\|_{L^2_{t,x}} +\|\partial^{\le 1}
v\|_{L^2_{t,x}(\R_+\times\{r\le 2\})}
\\\le \tilde{R}^{-\delta} \|\la r\ra^{\frac{1}{2}-s} r^{-1}
v\|_{L^2_{t,x}}
+ C_{\tilde{R}}\Bigl(\|\la r\ra^{-\frac{1}{2}}\partial w\|_{\ell^\infty_{R\le \tilde{R}}
  \ell^2_\tau L^2_{t,x}(C^R_\tau)}
+\|\la r\ra^{-\frac{3}{2}}w\|_{\ell^\infty_{R\le \tilde{R}}
  \ell^2_\tau L^2_{t,x}(C^R_\tau)}\Bigr).
  \end{multline*}
By selecting $\tilde{R}$ sufficiently large, we may absorb the first
term in the right into the left side of \eqref{rp6}.  The latter terms
may subsequently be bounded using \eqref{le} and \eqref{Fterm} to
obtain
\begin{multline}
  \label{rp6a}
  \|\la r\ra^{\frac{1}{2}-s} r^{-1}
v\|_{L^2_{t,x}}+
\|\la r\ra^{\frac{1}{2}-s}
  r^{-1}(\partial_t+\partial_r)(rv)\|_{L^2_{t,x}} +
  \|\la r\ra^{\frac{1}{2}-s}\ang v\|_{L^2_{t,x}}
  \\+ \|\la r\ra^{\frac{1}{2}-s+\sigma} \la t-r\ra^{-\sigma}
  r^{-1}(\partial_t+\partial_r)(r v)\|_{\ell^\frac{2}{1-2\sigma}_U\ell^2_\tau
    L^2_{t,x}(C^U_\tau)}
    + \|\la r\ra^{-\frac{1}{2}-s+\sigma} \la t-r\ra^{-\sigma}
   v\|_{\ell^\frac{2}{1-2\sigma}_U\ell^2_\tau
    L^2_{t,x}(C^U_\tau)}
  \\
  \lesssim \|\partial
  v(0,\cd)\|_{L^2}
  +\inf_{
    F=F_1+F_2}\Bigl(
 \|\la r\ra^{\frac{3}{2}-s} F_1\|_{\ell^2_\tau \ell^2_{R} L^2_{t,x}(C^R_\tau)}
 +\|\la r\ra^{1-s}\la t-r\ra^{\frac{1}{2}} F_2\|_{\ell^1_U \ell^2_{\tau} L^2_{t,x}(C^U_\tau)}\Bigr).
\end{multline}

We may now use
\eqref{v} to get a statement about $w$.  Since
$e^{-\frac{a}{2\delta}\la r\ra^{-\delta}\beta(r)}$ is bounded, this gives
\begin{multline}
  \label{rp7}
  \|\la r\ra^{\frac{1}{2}-s} r^{-1}
w\|_{L^2_{t,x}}+
\|\la r\ra^{\frac{1}{2}-s}
  r^{-1}(\partial_t+\partial_r)(r w)\|_{L^2_{t,x}} +
  \|\la r\ra^{\frac{1}{2}-s}\ang w\|_{L^2_{t,x}}
  \\+ \|\la r\ra^{\frac{1}{2}-s+\sigma} \la t-r\ra^{-\sigma}
  r^{-1}(\partial_t+\partial_r)(r
  w)\|_{\ell^\frac{2}{1-2\sigma}_U\ell^2_{\tau\ge 4U}
    L^2_{t,x}(C^U_\tau)}
    + \|\la r\ra^{-\frac{1}{2}-s+\sigma} \la t-r\ra^{-\sigma}
   w\|_{\ell^\frac{2}{1-2\sigma}_U\ell^2_{\tau\ge 4U}
    L^2_{t,x}(C^U_\tau)}
  \\
  \lesssim \|\partial
  w(0,\cd)\|_{L^2}
  +\inf_{F=F_1+F_2}\Bigl(
 \|\la r\ra^{\frac{3}{2}-s} F_1\|_{\ell^2_\tau \ell^2_{R} L^2_{t,x}(C^R_\tau)}
 +\|\la r\ra^{1-s}\la t-r\ra^{\frac{1}{2}} F_2\|_{\ell^1_U \ell^2_{\tau} L^2_{t,x}(C^U_\tau)}\Bigr)
 \\+ \|\la r\ra^{-\frac{1}{2}-s-\delta}
w\|_{L^2_{t,x}}
 +\|\la r\ra^{-\frac{1}{2}-s+\sigma-\delta} \la t-r\ra^{-\sigma}
  w\|_{\ell^\frac{2}{1-2\sigma}_U\ell^2_{\tau\ge 4U}
    L^2_{t,x}(C^U_\tau)}.
\end{multline}

We end the proof by addressing the error terms.  If we fix
$\tilde{R}$ sufficiently large, we may bound
\begin{multline}\label{rp8}
  ||\la r\ra^{-\frac{1}{2}-s-\delta}
w\|_{L^2_{t,x}}
 +\|\la r\ra^{-\frac{1}{2}-s+\sigma-\delta} \la t-r\ra^{-\sigma}
   w\|_{\ell^\frac{2}{1-2\sigma}_U\ell^2_{\tau\ge 4U}
    L^2_{t,x}(C^U_\tau)}
  \le
C \tilde{R}^{-\delta} \|\la r\ra^{-\frac{1}{2}-s}
w\|_{\ell^2_\tau \ell^2_{R\ge \tilde{R}} L^2_tL^2_x(C^R_\tau)}
 \\+C \tilde{R}^{-\delta}\|\la r\ra^{-\frac{1}{2}-s+\sigma} \la t-r\ra^{-\sigma}
  r w\|_{\ell^\frac{2}{1-2\sigma}_U\ell^2_{\tau\ge \max(4U,\tilde{R}/2)}
    L^2_{t,x}(C^U_\tau)}
  +C_{\tilde{R}}\|w\|_{\ell^2_\tau \ell^2_{R\le \tilde{R}} L^2_{t,x}(C^R_\tau)}.
\end{multline}
For $\tilde{R}$ sufficiently large, the first two terms in the right
side of \eqref{rp8} can be absorbed into the left side of
\eqref{rp7}.  The remaining term is supported on a fixed compact set
where the weights become irrelevant.  This is readily bounded using
\eqref{le} and \eqref{Fterm}, which yields the desired estimate.
\end{proof}

We also consider the higher order estimate where vector fields are
applied to the solution $w$.
\begin{corollary}
  Let $0<s\le 1/2$ and $0\le \sigma< s$.  Then if $w\in C^4(\R_+\times \R^3)$ solves
  \eqref{boxw} subject to \eqref{supp}, we have
  \begin{multline}\label{rpZ2}
   \|\la r\ra^{\frac{1}{2}-s} r^{-1}
Z^{\le 2} w\|_{L^2_{t,x}}+
\|\la r\ra^{\frac{1}{2}-s}
r^{-1}(\partial_t+\partial_r)(r Z^{\le 2} w)\|_{L^2_{t,x}}
+\|\la r\ra^{\frac{1}{2}-s} \ang Z^{\le 2}w\|_{L^2_{t,x}}
  \\+ \|\la r\ra^{\frac{1}{2}-s+\sigma} \la t-r\ra^{-\sigma}
  r^{-1}(\partial_t+\partial_r)(r
  Z^{\le 2} w)\|_{\ell^\frac{2}{1-2\sigma}_U\ell^2_{\tau\ge 4U}
    L^2_{t,x}(C^U_\tau)}
 \\   + \|\la r\ra^{-\frac{1}{2}-s+\sigma} \la t-r\ra^{-\sigma}
   Z^{\le 2} w\|_{\ell^\frac{2}{1-2\sigma}_U\ell^2_{\tau\ge 4U}
    L^2_{t,x}(C^U_\tau)}
  \\
  \lesssim \|\partial \nabla^{\le 2}
  w(0,\cd)\|_{L^2}
  +
 \|\la r\ra^{\frac{3}{2}-s} Z^{\le 2} F\|_{\ell^2_\tau \ell^2_{R\le \tau/2} L^2_{t,x}(C^R_\tau)}
 +\|\la r\ra^{1-s}\la t-r\ra^{\frac{1}{2}} Z^{\le 2} F\|_{\ell^1_U
   \ell^2_{\tau\ge 4U} L^2_{t,x}(C^U_\tau)}.
\end{multline}
\end{corollary}

\begin{proof}
With $\L$ as in \eqref{L}, we first compute
\begin{multline}\label{commS}
  [\L,S] = 2\L + a\frac{\delta r - 1}{\la r\ra^{2+\delta}}\beta(r)
    \Bigl(\partial_t+\partial_r+\frac{1}{r}\Bigr)
    + \frac{a(1+\delta)}{2} \frac{2-\delta r}{\la r\ra^{3+\delta}} \beta(r)
    + \frac{a^2}{2} \frac{1-\delta r}{\la r\ra^{3+2\delta}} \beta(r)^2 \\+ \tilde{V}_{S,0}(r)\partial_t
    + \tilde{V}_{S,1}(r) \partial_r + \tilde{V}_S(r)
\end{multline}
where $\tilde{V}_S$, $\tilde{V}_{S,0}$, and $\tilde{V}_{S,1}$ are
radial $C^\infty([0,\infty))$-functions that are supported in $\{r\le 2\}$.  Additionally,
using \eqref{angcomm}, we
record that
\begin{multline}
  \label{commd}
  [\L,\nabla_{x_j}] = \omega_j \frac{a(1+\delta)}{\la
    r\ra^{2+\delta}}\beta(r) \Bigl(\partial_t+\partial_r+\frac{1}{r}\Bigr)
  -\frac{a}{r\la r\ra^{1+\delta}} \beta(r) \ang_j + \omega_j\frac{a}{\la r\ra^{1+\delta}
    r^2} \beta(r) - \omega_j \frac{a(1+\delta)(2+\delta)}{2\la
    r\ra^{3+\delta}}\beta(r)
  \\-\omega_j \frac{a^2(1+\delta)}{2\la r\ra^{3+2\delta}} (\beta(r))^2
  +\omega_j \tilde{V}_{\partial,0}(r)\partial_t + \omega_j \tilde{V}_{\partial,
    1}(r)\partial_r + \tilde{V}_{\partial,2}(r)\ang_j
  + \omega_j \tilde{V}_{\partial}(r).
\end{multline}
Here $\tilde{V}_{\partial}$, $\tilde{V}_{\partial,1}$,
and $\tilde{V}_{\partial,2}$ denote radial $C^\infty([0,\infty))$-functions that are
supported in $\{1\le r\le 2\}$, while $\tilde{V}_{\partial,0}$ is a
smooth (in $r$) function that is supported within $\{r\le 2\}$.  Even in the
latter case, this yields coefficients that are $C^1(\R^3)$, which is
sufficient for what follows.

We 
apply \eqref{rp6a} to $Zv$, which causes $F$ in the right side to be replaced by
$e^{\frac{a}{2\delta}\la r\ra^{-\delta}\beta(r)} \Bigl(Z\mathcal{L}v +
[\mathcal{L},Z]v\Bigr)$.  Here we include
the commutators above in $F_1$.  We then
see that
\begin{multline}
  \label{rpZ11}
  \|\la r\ra^{\frac{1}{2}-s} r^{-1}
Z v\|_{L^2_{t,x}}+
\|\la r\ra^{\frac{1}{2}-s}
  r^{-1}(\partial_t+\partial_r)(r Zv)\|_{L^2_{t,x}} +
  \|\la r\ra^{\frac{1}{2}-s}\ang Z v\|_{L^2_{t,x}}
  \\+ \|\la r\ra^{\frac{1}{2}-s+\sigma} \la t-r\ra^{-\sigma}
  r^{-1}(\partial_t+\partial_r)(r Zv)\|_{\ell^\frac{2}{1-2\sigma}_U\ell^2_\tau
    L^2_{t,x}(C^U_\tau)}
    + \|\la r\ra^{-\frac{1}{2}-s+\sigma} \la t-r\ra^{-\sigma}
  Z v\|_{\ell^\frac{2}{1-2\sigma}_U\ell^2_\tau
    L^2_{t,x}(C^U_\tau)}
  \\
  \lesssim \|\partial
  Z v(0,\cd)\|_{L^2}
  +\inf_{F=F_1+F_2}\Bigl(
 \|\la r\ra^{\frac{3}{2}-s} 
 Z^{\le 1} F_1\|_{\ell^2_\tau \ell^2_{R} L^2_{t,x}(C^R_\tau)}
 +\|\la r\ra^{1-s}\la t-r\ra^{\frac{1}{2}}
 Z^{\le 1}F_2\|_{\ell^1_U \ell^2_{\tau} L^2_{t,x}(C^U_\tau)}\Bigr)
\\ +\|\la r\ra^{\frac{1}{2}-s-\delta}
 r^{-1}(\partial_t+\partial_r)(rv)\|_{L^2_{t,x}}
 + \|\la r\ra^{-\frac{1}{2}-s-\delta} v\|_{L^2_{t,x}} \\+ \|\la
 r\ra^{-\frac{1}{2}-s-\delta} \ang v\|_{L^2_{t,x}}
 + \|\partial^{\le 1} v\|_{L^2_{t,x}(\R_+\times \{r\le 2\})}.
\end{multline}
We may apply \eqref{rp6a} and \eqref{le} to the last four terms in the
right side, which result from commuting the vector fields with
$\mathcal{L}$.  This gives
\begin{multline}
  \label{rpZ12}
  \|\la r\ra^{\frac{1}{2}-s} r^{-1}
Z^{\le 1} v\|_{L^2_{t,x}}+
\|\la r\ra^{\frac{1}{2}-s}
  r^{-1}(\partial_t+\partial_r)(r Z^{\le 1}v)\|_{L^2_{t,x}} +
  \|\la r\ra^{\frac{1}{2}-s}\ang Z^{\le 1} v\|_{L^2_{t,x}}
  \\+ \|\la r\ra^{\frac{1}{2}-s+\sigma} \la t-r\ra^{-\sigma}
  r^{-1}(\partial_t+\partial_r)(r Z^{\le 1} v)\|_{\ell^\frac{2}{1-2\sigma}_U\ell^2_\tau
    L^2_{t,x}(C^U_\tau)}
    + \|\la r\ra^{-\frac{1}{2}-s+\sigma} \la t-r\ra^{-\sigma}
   Z^{\le 1} v\|_{\ell^\frac{2}{1-2\sigma}_U\ell^2_\tau
    L^2_{t,x}(C^U_\tau)}
  \\
  \lesssim \|\partial
  \nabla^{\le 1} v(0,\cd)\|_{L^2}
  +\inf_{F=F_1+F_2}\Bigl(
 \|\la r\ra^{\frac{3}{2}-s}
 Z^{\le 1} F_1\|_{\ell^2_\tau \ell^2_{R} L^2_{t,x}(C^R_\tau)}
 +\|\la r\ra^{1-s}\la t-r\ra^{\frac{1}{2}}
 Z^{\le 1}F_2\|_{\ell^1_U \ell^2_{\tau} L^2_{t,x}(C^U_\tau)}\Bigr).
\end{multline}
Subsequently applying \eqref{rpZ12} to $Zv$ yields
\begin{multline}
  \label{rpZ21}
  \|\la r\ra^{\frac{1}{2}-s} r^{-1}
Z^{\le 1} Z v\|_{L^2_{t,x}}+
\|\la r\ra^{\frac{1}{2}-s}
  r^{-1}(\partial_t+\partial_r)(r Z^{\le 1} Zv)\|_{L^2_{t,x}} +
  \|\la r\ra^{\frac{1}{2}-s}\ang Z^{\le 1} Z v\|_{L^2_{t,x}}
  \\+ \|\la r\ra^{\frac{1}{2}-s+\sigma} \la t-r\ra^{-\sigma}
  r^{-1}(\partial_t+\partial_r)(r Z^{\le 1} Zv)\|_{\ell^\frac{2}{1-2\sigma}_U\ell^2_\tau
    L^2_{t,x}(C^U_\tau)}
    + \|\la r\ra^{-\frac{1}{2}-s+\sigma} \la t-r\ra^{-\sigma}
  Z^{\le 1} Z v\|_{\ell^\frac{2}{1-2\sigma}_U\ell^2_\tau
    L^2_{t,x}(C^U_\tau)}
  \\
  \lesssim \|\partial
  \nabla^{\le 2} v(0,\cd)\|_{L^2}
  +
 \|\la r\ra^{\frac{3}{2}-s} 
 Z^{\le 2} F\|_{\ell^2_\tau \ell^2_{R\le
     \tau/2} L^2_{t,x}(C^R_\tau)}
 +\|\la r\ra^{1-s}\la t-r\ra^{\frac{1}{2}}
Z^{\le 2}F\|_{\ell^1_U \ell^2_{\tau\ge
     4U} L^2_{t,x}(C^U_\tau)}
\\ +\|\la r\ra^{\frac{1}{2}-s-\delta}
 r^{-1}(\partial_t+\partial_r)(r Z^{\le 1} v)\|_{L^2_{t,x}}
 + \|\la r\ra^{-\frac{1}{2}-s-\delta} Z^{\le 1} v\|_{L^2_{t,x}} \\+ \|\la
 r\ra^{-\frac{1}{2}-s-\delta} \ang Z^{\le 1} v\|_{L^2_{t,x}}
 + \|\partial^{\le 1} Z^{\le 1} v\|_{L^2_{t,x}(\R_+\times \{r\le 2\})}
 + \|r^{-1} \partial_t v\|_{L^2_{t,x}(\R_+\times\{r\le 2\})}.
\end{multline}
Here we have included the commutator terms in $F_1$ but have otherwise
made the selection \eqref{Fsplit}.  The last term results from the
vector field within $\|\la r\ra^{\frac{3}{2}-s} Z^{\le 1}
F_1\|_{\ell^2_\tau\ell^2_RL^2_{t,x}(C^R_\tau)}$ landing on the
$\omega_j$ that precedes $\tilde{V}_{\partial,0}\partial_t$ in
\eqref{commd}.  As the support of $\tilde{V}_{\partial,0}$ may contain
$r=0$, we must take some care with the singularity.  Since we may
apply \eqref{L2Hardy} to $(1-\beta(r/2))\partial_t v$, we may bound
\[\|r^{-1} \partial_t v\|_{L^2_{t,x}(\R_+\times\{r\le 2\})}
  \lesssim \|\partial_t \nabla v\|_{L^2_{t,x}(\R_+\times\{r\le 4\})}
  + \|\partial_t v\|_{L^2_{t,x}(\R_+\times\{r\le 4\})}.\]
Thus, encorporating this, it remains to bound the error terms
\begin{multline*}\|\la r\ra^{\frac{1}{2}-s-\delta}
 r^{-1}(\partial_t+\partial_r)(r Z^{\le 1} v)\|_{L^2_{t,x}}
 + \|\la r\ra^{-\frac{1}{2}-s-\delta} Z^{\le 1} v\|_{L^2_{t,x}} \\+ \|\la
 r\ra^{-\frac{1}{2}-s-\delta} \ang Z^{\le 1} v\|_{L^2_{t,x}}
 + \|\partial^{\le 1} Z^{\le 1} v\|_{L^2_{t,x}(\R_+\times \{r\le 4\})}\end{multline*}
from \eqref{rpZ21}.
These can be immediately handled using \eqref{rpZ12},
\eqref{leZ2}, and \eqref{Fterm} giving
\begin{multline}
  \label{rpZ22}
  \|\la r\ra^{\frac{1}{2}-s} r^{-1}
Z^{\le 2}v\|_{L^2_{t,x}}+
\|\la r\ra^{\frac{1}{2}-s}
  r^{-1}(\partial_t+\partial_r)(r Z^{\le 2}v)\|_{L^2_{t,x}} +
  \|\la r\ra^{\frac{1}{2}-s}\ang Z^{\le 2} v\|_{L^2_{t,x}}
  \\+ \|\la r\ra^{\frac{1}{2}-s+\sigma} \la t-r\ra^{-\sigma}
  r^{-1}(\partial_t+\partial_r)(r Z^{\le 2} v)\|_{\ell^\frac{2}{1-2\sigma}_U\ell^2_\tau
    L^2_{t,x}(C^U_\tau)}
    + \|\la r\ra^{-\frac{1}{2}-s+\sigma} \la t-r\ra^{-\sigma}
  Z^{\le 2}v\|_{\ell^\frac{2}{1-2\sigma}_U\ell^2_\tau
    L^2_{t,x}(C^U_\tau)}
  \\
  \lesssim \|\partial
  \nabla^{\le 2} v(0,\cd)\|_{L^2}
  +
 \|\la r\ra^{\frac{3}{2}-s} 
 Z^{\le 2} F\|_{\ell^2_\tau \ell^2_{R\le
     \tau/2} L^2_{t,x}(C^R_\tau)}
 +\|\la r\ra^{1-s}\la t-r\ra^{\frac{1}{2}}
Z^{\le 2}F\|_{\ell^1_U \ell^2_{\tau\ge
     4U} L^2_{t,x}(C^U_\tau)}.\end{multline}

To complete the proof, we need only to pass back to a statement for
$w$.  This follows exactly as in the proof of Theorem \ref{thm_rp}.
We first use the analog of \eqref{rp8} to get smallness outside of a
compact region, and we subsequently apply \eqref{leZ2} and
\eqref{Fterm} to bound the remaining error terms.
\end{proof}

\vskip .25in

\section{Decay estimates}

The space-time decay estimates from \cite[Lemma 3.8]{MTT} are:
\begin{proposition}\label{prop_mtt}
  For $\tau \ge 1$, $1\le R\le \tau/2$, and $1\le U\le \tau/4$, we have
  the following space-time Klainerman-Sobolev estimates:
  \begin{equation}
    \label{inftyR}
    \|w\|_{L^\infty_{t,x}(C^R_\tau)} \lesssim
    \tau^{-\frac{1}{2}}R^{-\frac{3}{2}} \|Z^{\le 3}
    w\|_{L^2_{t,x}(\tilde{C}^R_\tau)} +
    \tau^{-\frac{1}{2}}R^{-\frac{1}{2}} \|r^{-1}(\partial_t+\partial_r)(r
    Z^{\le 2} w)\|_{L^2_{t,x}(\tC^R_\tau)},
  \end{equation}
\begin{equation}
  \label{inftyU}
    \|w\|_{L^\infty_{t,x}(C^U_\tau)} \lesssim
    \tau^{-\frac{3}{2}}U^{-\frac{1}{2}} \|Z^{\le 3}
    w\|_{L^2_{t,x}(\tC^U_\tau)} + \tau^{-\frac{1}{2}}U^{-\frac{1}{2}}
    \| r^{-1}(\partial_t+\partial_r)(r Z^{\le 2}w)\|_{L^2_{t,x}(\tC^U_\tau)}.
  \end{equation}
\end{proposition}

After localizing with cutoff functions and converting to $t=e^s$ and
$r=e^{s+\rho}$ in the $C^R_\tau$ regions and $t=e^s$, $t-r=e^{s+\rho}$
in the $C^U_\tau$ regions, the estimate follows from standard Sobolev
bounds in the $(s,\rho,\omega)$ coordinates.  Here, it is important to
note, e.g.,
\[\partial_s (w(e^s, e^{s+\rho}\omega)) = (Sw)(e^s,
  e^{s+\rho}\omega),\quad \partial_\rho(w(e^s, e^{s+\rho}\omega)) =
  (r\partial_r)w(e^s,e^{s+\rho}\omega)\]
on the $C^R_\tau$ regions.  A similar calculation is made for
$C^U_\tau$.  We have additionally used that
\[\partial_t-\partial_r = \frac{2}{t-r} S -
  \frac{t+r}{t-r}(\partial_t+\partial_r)\]
to convert the $\partial_r = \frac{1}{2}\Bigl[(\partial_t+\partial_r)
+ (\partial_t-\partial_r)\Bigr]$ that occurs into the ``good''
derivative $\partial_t+\partial_r$.  This conversion costs an
additional vector field on the first term in the right when compared to \cite{MTT}.
See, also, \cite{Looi-Tohaneanu},
\cite{M-Rhoads}, \cite{M-Stewart}.  The nonlinearity under
consideration is not sufficiently regular to permit a loss of three
vector fields.  So we prove alternatives to \eqref{inftyR},
\eqref{inftyU} that will be used instead.

The new decay bounds, which notably do not
lose three vector fields in the right, are:
\begin{theorem}
  For $\tau \ge 1$, $1\le R\le \tau/2$, and $1\le U\le \tau/4$, we
  have
  \begin{multline}
    \label{newCRT}
        \|w\|_{L^\infty_{t,x}(C^R_\tau)} \lesssim
    \frac{1}{\tau^{\frac{1}{2}}R^{\frac{3}{2}}} \|Z^{\le 2}
    w\|_{L^2_{t,x}(\tC^R_\tau)} \\+
    \frac{1}{\tau^{\frac{1}{2}}R^{\frac{1}{2}}}
    \|r^{-1}(\partial_t+\partial_r)(r Z^{\le 2}
    w)\|_{L^2_{t,x}(\tC^R_\tau)}
    +
    \frac{1}{\tau^{\frac{1}{2}}R^{\frac{1}{2}}}
    \|\ang (Z^{\le 2} w)\|_{L^2_{t,x}(\tC^R_\tau)},
  \end{multline}
  and
    \begin{equation}
    \label{newCUT}
        \|w\|_{L^\infty_{t,x}(C^U_\tau)} \lesssim
    \frac{1}{\tau^{\frac{3}{2}}U^{\frac{1}{2}}} \|Z^{\le 2}
    w\|_{L^2_{t,x}(\tC^U_\tau)} +
    \frac{1}{\tau^{\frac{1}{2}}U^{\frac{1}{2}}}
    \|r^{-1}(\partial_t+\partial_r)(rZ^{\le 2} w)\|_{L^2_{t,x}(\tC^U_\tau)}.
  \end{equation}

  Moreover, for $2\le q\le 4$, we have
  \begin{multline}
    \label{LqCRT}
    \|w\|_{L^q_{t,x}(C^R_\tau)}\lesssim
    \tau^{\frac{1}{q}-\frac{1}{2}}R^{\frac{3}{q}-\frac{3}{2}} \|Z^{\le
      1} w\|_{L^2_{t,x}(\tC^R_\tau)} \\+
    \tau^{\frac{1}{q}-\frac{1}{2}}R^{\frac{3}{q}-\frac{1}{2}}
    \|r^{-1}(\partial_t+\partial_r)(rw)\|_{L^2_{t,x}(\tC^R_\tau)}
    +  \tau^{\frac{1}{q}-\frac{1}{2}}R^{\frac{3}{q}-\frac{1}{2}}
    \|\ang w\|_{L^2_{t,x}(\tC^R_\tau)}
  \end{multline}
  and
\begin{equation}
    \label{LqCUT}
    \|w\|_{L^q_{t,x}(C^U_\tau)}\lesssim
    U^{\frac{1}{q}-\frac{1}{2}}\tau^{\frac{3}{q}-\frac{3}{2}} \|Z^{\le
      1} w\|_{L^2_{t,x}(\tC^U_\tau)} +
    U^{\frac{1}{q}-\frac{1}{2}}\tau^{\frac{3}{q}-\frac{1}{2}}
    \|r^{-1}(\partial_t+\partial_r)(rw)\|_{L^2_{t,x}(\tC^U_\tau)}.
  \end{equation}
\end{theorem}

\begin{proof}
  To prove \eqref{newCRT} and \eqref{newCUT}, we consider a few cases separately.
\vskip .125in
  \noindent{\bf Case 1  ($\tau > 8$, $1<R\le \tau/2$):}
  We introduce the coordinates
  \[t-r=e^s,\quad r=e^{s+\rho},\quad \phi(s,\rho,\omega)=(e^s+e^{s+\rho},e^{s+\rho}\omega)\]
  and notice that
  \[\partial_s(w(e^s+e^{s+\rho}, e^{s+\rho}\omega)) =
    (Sw)(e^s+e^{s+\rho}, e^{s+\rho}\omega),
    \quad
    \partial_\rho(w(e^s+e^{s+\rho}, e^{s+\rho}\omega)) =
    (r(\partial_t+\partial_r)w)(e^s+e^{s+\rho}, e^{s+\rho}\omega).\]
  Thus if we include cutoffs to localize to $C^R_\tau$ and apply
  Sobolev embeddings in $(s,\rho,\omega)$, we obtain
  \begin{align*}\|w\|_{L^\infty_{t,x}(C^R_\tau)}&\lesssim
                                       \|\partial_{s,\omega}^{\le 2}
                                       w\|_{L^2_{s,\rho,\omega}(\phi^{-1}(\tC^R_\tau))}
                                       + \|\partial_\rho
                                       \partial_{s,\omega}^{\le 2}
                                       w\|_{L^2_{s,\rho,\omega}(\phi^{-1}(\tC^R_\tau))}\\
    &\lesssim R^{-\frac{3}{2}}\tau^{-\frac{1}{2}} \|Z^{\le 2}
      w\|_{L^2_{t,x}(\tC^R_\tau)}
      + R^{-\frac{1}{2}}\tau^{-\frac{1}{2}}
      \|(\partial_t+\partial_r)Z^{\le 2} w\|_{L^2_{t,x}(\tC^R_\tau)}.
  \end{align*}
  Here, we have used that $t-r\approx \tau$ on $C^R_\tau$.
  After noticing
  \[(\partial_t+\partial_r)w=r^{-1}(\partial_t+\partial_r)(rw) -
    \frac{1}{r}w,\]
  we obtain \eqref{newCRT}.
   For these
  cases, we do not require the last term in the right side.

\vskip .125in
\noindent{\bf Case 2  ($\tau > 8$, $1<U\le \tau/4$):}
The second case is quite similar to the preceding one.  After
localizing to $C^U_\tau$ and using the same coordinates, Sobolev
estimates in $(s,\rho,\omega)$ gives \eqref{newCUT}.  We are now
noticing that $r\approx \tau$ on $C^U_\tau$ in the given ranges of the parameters.

\vskip .125in
In these first two cases, we have only relied upon the homogeneous
vector fields $S$ and $\Omega$.

\vskip .125in
\noindent{\bf Case 3  ($\tau > 8$, $R=1$):}
Considering
\[t=e^s + |y|,\quad x=y,\quad \phi(s,y)=(e^s+|y|,y)\]
we note
\begin{equation}\label{ds}\partial_s (w(e^s+|y|,y)) = ((t-r)\partial_t w)(e^s+|y|,y)
  = \Bigl(\frac{t-r}{t}Sw\Bigr)(e^s+|y|,y) -
  \Bigl(\frac{t-r}{t}r\partial_rw\Bigr)(e^s+|y|,y).
\end{equation}
Applying a $1$-dimensional Sobolev embedding in $s$ and using that
$\frac{t-r}{t}$ is bounded on $C^{R=1}_\tau$, we have
\[\|w\|_{L^\infty_{t,x}(C^{R=1}_\tau)} \lesssim \sum_{|\alpha|+j\le 1}
  \|\beta_{C^{R=1}_\tau}(\phi(s,y)) (S^j\partial_x^{|\alpha|}
  w)(\phi(s,y))\|_{L^\infty_y L^2_s},
\]
where $\beta_{C^{R=1}_\tau}$ is an appropriate cutoff that is
identically $1$ on an enlargement of $C^{R=1}_\tau$ but vanishes
outside $\tC^{R=1}_\tau$.

We want to subsequently apply an embedding in the $y$-variables.  This
gives
\begin{multline*}\|w\|_{L^\infty_{t,x}(C^{R=1}_\tau)} \lesssim
  \sum_{\substack{|\alpha|+j+|\mu|\le 2\\|\alpha|+j\le 1}}
  \|\partial^{|\mu|}_y \beta_{C^{R=1}_\tau}(\phi(s,y))
  (S^j\partial_x^{|\alpha|} w)(\phi(s,y))\|_{L^2_{s,y}}
\\+ \sum_{|\mu|=2} \sum_{j+|\alpha|=1}\|\partial_y^{|\mu|} \beta_{C^{R=1}_\tau}(\phi(s,y))
  (S^j\partial_x^{|\alpha|} w)(\phi(s,y))\|_{L^2_{s,y}}.
\end{multline*}
We note that
\begin{equation}\label{dy}\begin{split}\partial_{y_j}(w(e^s+|y|,y)) &=
  \Bigl(\frac{x_j}{r}(\partial_t+\partial_r)w\Bigr)(e^s+|y|,y)
  +\ang_j w (e^s+|y|,y)\\
&=\Bigl[\Bigl(\frac{x_j}{r} \frac{1}{t}S - \frac{x_j}{t}\partial_r+ \partial_{x_j}\Bigr)w\Bigr](e^s+|y|,y).
  \end{split}
\end{equation}
We need only observe that the coefficients are bounded and that
$\partial_{y_j}\partial_{y_k} (w(\phi(s,y))) = (Yw)(\phi(s,y))$ with
\begin{multline*}Y
  =-\omega_j\omega_k \frac{1}{t^2}S + \frac{1}{tr}
  (\delta_{jk}-\omega_j\omega_k)S
  -\frac{1}{t}(\delta_{jk}-\omega_j\omega_k)\omega_l \partial_{x_l}
  -\frac{1}{t}\omega_k (\delta_{jl}-\omega_j\omega_l)\partial_{x_l}
  -\frac{t-r}{t^2}\omega_j\omega_k\omega_l \partial_{x_l}
  \\+\omega_j\omega_k \frac{1}{t}(\partial_t+\partial_r)S - \frac{r}{t}
  \omega_j\omega_k\omega_l (\partial_t+\partial_r)\partial_{x_l}
  + \omega_j (\partial_t+\partial_r)\partial_{x_k}
  +\omega_k \frac{1}{t}\ang_j S
  - \frac{r}{t} \omega_k\omega_l \ang_j \partial_{x_l} + \ang_j \partial_k.
  \end{multline*}
Thus, upon changing back to $(t,x)$ coordinates and using that
$t-r\approx \tau$ on $\tC^{R=1}_\tau$, we have
\[\|w\|_{L^\infty_{t,x}(C^{R=1}_\tau)}
  \lesssim \tau^{-\frac{1}{2}} \|r^{-1} Z^{\le 2}
  w\|_{L^2_{t,x}(\tC^{R=1}_\tau)}
  + \tau^{-\frac{1}{2}} \|r^{-1}(\partial_t+\partial_r)(r Z^{\le 2}
  w)\|_{L^2_{t,x}(\tC^{R=1}_\tau)}
  + \tau^{-\frac{1}{2}}\|\ang Z^{\le 2} w\|_{L^2_{t,x}(\tC^{R=1}_\tau)}.\]
  Here we have also written
  \[(\partial_t+\partial_r)w = r^{-1}(\partial_t+\partial_r)(rw) -
    \frac{1}{r}w.\]
  To resolve the singular contribution of the first term on the right,
  we first note that there is nothing required on
  $\tC^{R=1}_\tau\backslash C^{R=1}_\tau$.  Then to
  $\beta_{C^{R=1}_\tau} Z^{\le 2} w$, which vanishes at $t=0$, we apply
  \eqref{hardyR}, with, say, $s=1/2$.  This gives
  \begin{align*}
    \tau^{-\frac{1}{2}}\|r^{-1} Z^{\le 2} w\|_{L^2_{t,x}(C^{R=1}_\tau)}
    &\lesssim \tau^{-\frac{1}{2}}\|r^{-1}(\partial_t+\partial_r)(r
      \beta_{C^{R=1}_\tau}Z^{\le 2}
      w)\|_{L^2_{t,x}}\\
    &\lesssim \tau^{-\frac{1}{2}} \|Z^{\le 2}
      w\|_{L^2_{t,x}(\tC^{R=1}_\tau)} +
      \tau^{-\frac{1}{2}}\|r^{-1}(\partial_t+\partial_r)(rZ^{\le 2} w)\|_{L^2_{t,x}(\tC^{R=1}_\tau)},
  \end{align*}
  which suffices and completes the proof of the third case.

We note that it is this case that necessitates the inclusion of the
angular derivative terms in the right side of \eqref{newCRT}.  We also
note that, for this case, we include the generators of spatial
translations in our admissible vector fields.  Notably, however, we do
not require $\partial_t$ as a vector field.

\vskip .125in
\noindent{\bf Case 4  ($\tau > 8$, $U=1$):}
Here we consider
\[t=q+e^\rho,\quad r=e^\rho.\]
If we now set $\phi(q,\rho,\omega) = (q+e^\rho,e^\rho\omega)$, 
Then
\[\partial_\rho (w(\phi(q,\rho\omega))) =
  (r(\partial_t+\partial_r)w)(\phi(q,\rho\omega)),\quad
  \partial_q(w(\phi(q,\rho,\omega))) =
  (\partial_tw)(\phi(q,\rho\omega)).\]
We moreover use that
\[\partial_t = \frac{1}{t}S - \frac{r}{t}\partial_r.\]
Applying Sobolev's estimate in $(q,\rho,\omega)$, after localizing to
$C^{U=1}_\tau$, gives
\begin{align*}\|w\|_{L^\infty_{t,x}(C^{U=1}_\tau)}
  &\lesssim \|\partial_{q,\omega}^{\le 2}
  w\|_{L^2_{q,\rho,\omega}(\phi^{-1}(C^{U=1}_\tau))}
  + \|\partial_\rho \partial_{q,\omega}^{\le 2}
    w\|_{L^2_{q,\rho,\omega}(\phi^{-1}(C^{U=1}_\tau))}\\
  &\lesssim \tau^{-\frac{3}{2}} \|Z^{\le 2}
    w\|_{L^2_{t,x}(\tC^{U=1}_\tau)}
    + \tau^{-\frac{1}{2}}\|r^{-1}(\partial_t+\partial_r)(rZ^{\le 2}w)\|_{L^2_{t,x}(\tC^{U=1}_\tau)}
  \end{align*}
as desired.

\vskip .125in
\noindent{\bf Case 5 ($\tau \le 8$):}
Letting $\beta$ be a smooth cutoff that is equal to $1$ on
$C_{\tau}$ and vanishes outside $\tilde{C}_{\tau}$,
we observe that
\[\frac{d}{dt}\|\beta v(t,\cd)\|^2_{L^2} = 2 \la \beta v(t,\cd),
  (\partial_t+\partial_r+\frac{1}{r})(\beta v)\ra.\]
We select $t_0$ so that
\[\|\beta v(t_0,\cd)\|^2_{L^2} \lesssim \|\beta v\|^2_{L^2_{t,x}(\tC_\tau)}.\]
Thus, if $t\in [\tau, 2\tau]$,
\begin{align*}\|\beta v(t,\cd)\|^2_{L^2} &\lesssim \|\beta
                                           v(t_0,\cd)\|^2_{L^2} +
                                           \|\beta
                                           v\|_{L^2_{t,x}(\tC_\tau)}\|r^{-1}(\partial_t+\partial_r)(r\beta
                                           v)\|_{L^2_{t,x}(\tC_\tau)}\\
  &\lesssim \|v\|^2_{L^2_{t,x}(\tC_\tau)} +
    \|r^{-1}(\partial_t+\partial_r)(r v)\|^2_{L^2_{t,x}(\tC_\tau)}.
\end{align*}
We combine the above with Sobolev's estimate to see
\begin{align*}
  \|w(t,\cd)\|_{L^\infty_x} &\lesssim \|\nabla^{\le 2}
                              w(t,\cd)\|_{L^2}\\
  &\lesssim \|\nabla^{\le 2} w\|_{L^2_{t,x}(\tC_\tau)} +
    \|r^{-1}(\partial_t+\partial_r)(r\nabla^{\le 2}w)\|_{L^2_{t,x}(\tC_\tau)},
\end{align*}
which completes the proof of \eqref{newCRT} and \eqref{newCUT}.

The argument for \eqref{LqCRT} and \eqref{LqCUT} is similar.  The same
changes of coordinates are used, but one instead uses $L^2-L^q$
embeddings in the new coordinates rather than the $L^2-L^\infty$ ones above.
\end{proof}

\vskip .25in

\section{Proof of Theorem \ref{main_thm}}
For such a small data result, the most delicate behavior occurs for
small $p$.  We assume that $1+\sqrt{2}<p<3$, and we set
\[s=\frac{3}{2}-\frac{2}{p-1},\quad \sigma=\frac{p-2}{2p}.\]
  Here we also notice that $\frac{2}{1-2\sigma}=p$ and $s< 1/2$
  for the given range of $p$.  Importantly, we
  observe that
  \[0\le \sigma < s \iff p > 1+\sqrt{2}.\]
  
  If we are in the case that
  $p\ge 3$, we may simply select $\tilde{p}$ satisfying the above.  We
  run the argument below using $\tilde{p}$ and may rely upon
  \eqref{newCRT} and \eqref{newCUT} to control the $p-\tilde{p}$ extra copies of
  the $w$.

We abbreviate
\begin{multline*}
  \|w\|_{X} =
\|\la r\ra^{\frac{1}{2}-s+\sigma} \la t-r\ra^{-\sigma}
    r^{-1}(\partial_t+\partial_r)(rw)\|_{\ell^{p}_U
      \ell^2_{\tau \ge 4U} L^2_{t,x}(C^U_\tau)} + \|\la
    r\ra^{-\frac{1}{2}-s+\sigma} \la t-r\ra^{-\sigma}
    w\|_{\ell^{p}_U \ell^2_{\tau\ge 4U}
      L^2_{t,x}(C^U_\tau)}
\\      + \|\la r\ra^{\frac{1}{2}-s}
r^{-1}(\partial_t+\partial_r)(rw)\|_{L^2_{t,x}}
+\|\la r\ra^{\frac{1}{2}-s} \ang w\|_{L^2_{t,x}}
    + \|\la r\ra^{-\frac{1}{2}-s} w\|_{L^2_{t,x}}.
\end{multline*}

We shall solve \eqref{mainEquation} via iteration.  Let $w_{0}\equiv
0$, and take $w_j$ to solve
\begin{equation}
  \label{iteration}
  \begin{cases}
    \Bigl(\Box + a \la r\ra^{-1-\delta}\partial_t\Bigr)w_j =
    F_p(w_{j-1}),\\
    w_j(0,\cd)=\varepsilon w_{(0)},\quad \partial_t
    w_j(0,\cd)=\varepsilon w_{(1)}.
  \end{cases}
\end{equation}
Here $w_{(0)}, w_{(1)}\in C^\infty_c(\R^3)$, and we take without loss
of generality $\supp\, w_{(0)}, w_{(1)} \subseteq \{|x|\le
\frac{1}{2}\}$.   

In Section \ref{section_bdd}, we will show that $(w_j)$ is uniformly
bounded in $\|Z^{\le 2} \cdot \|_{X}$.  In the subsequent section,
Section \ref{section_Cauchy}, we shall show that the sequence is
Cauchy, and thus converges, in $X$.  By standard arguments, the limit
of this sequence must be the desired solution.

\subsection{Multilinear estimates}
The multilinear estimates that are required are similar to
\cite[(5.6) and (5.7)]{MW}, though we will be working in a different
$\|\,\cd\,\|_{N_0}$ norm that corresponds to the right side of
\eqref{rpZ2}.

There are four cases to consider, and we
will seek to prove
\begin{align}
  \label{UZ2}
\|\la r\ra^{1-s} \la t-r\ra^{\frac{1}{2}}
  |g|^{p-1} f\|_{\ell^1_U \ell^2_{\tau\ge 4U} L^2_tL^2_{x}(C^U_\tau)}
  &\lesssim \|Z^{\le 2} g\|^{p-1}_X \|f\|_X,\\
  \label{RZ2} \|\la r\ra^{\frac{3}{2}-s}|g|^{p-1} f\|_{\ell^2_\tau \ell^2_{R\le \tau/2}
    L^2_tL^2_{x}(C^R_\tau)}&\lesssim \|Z^{\le 2} g\|^{p-1}_X \|f\|_X,\\
  \label{UZ1}\|\la r\ra^{1-s} \la t-r\ra^{\frac{1}{2}}
  |g|^{p-2} f^2\|_{\ell^1_U \ell^2_{\tau\ge 4U}
  L^2_tL^2_{x}(C^U_\tau)}&\lesssim \|Z^{\le 2}g\|^{p-2}_X \|Z^{\le 1}f\|^2_X,\\
  \label{RZ1}\|\la r\ra^{\frac{3}{2}-s} |g|^{p-2}f^2\|_{\ell^2_\tau \ell^2_{R\le \tau/2}
    L^2_tL^2_{x}(C^R_\tau)}&\lesssim  \|Z^{\le 2}g\|^{p-2}_X \|Z^{\le 1}f\|^2_X.
\end{align}

\vskip .125in
\subsubsection{Proof of \eqref{UZ2}}
We begin with examining the left side of \eqref{UZ2}.
Using H\"older's inequality, for fixed $\tau, 1\le U \le \tau/4$, we get the bound
\[
\Bigl\|\tau^{1-s} U^{\frac{1}{2}}
  |g|^{p-1} f\Bigr\|_{L^2_{t,x}(C^U_\tau)}\\\lesssim
   \Bigl\|\tau^{\frac{3-2\sigma}{2(p-1)}} U^{\frac{1+2\sigma}{2(p-1)}}
     g\Bigr\|^{p-1}_{L^\infty_{t,x}(C^U_\tau)}
     \Bigl\|\la r\ra^{-\frac{1}{2}-s+\sigma}\la t-r\ra^{-\sigma} 
     f\Bigr\|_{L^2_{t,x}(C^U_\tau)}.
 \]

  To the first factor in the right, we apply \eqref{newCUT}, which
  bounds the right side by
  \begin{multline*}
   \Bigl(\Bigl\|\la r\ra^{-\frac{1}{2}-s+\sigma} \la t-r\ra^{-\sigma}
  Z^{\le 2} g\Bigr\|^{p-1}_{L^2_{t,x}(\tC^U_{t,x})}
   +    \Bigl\|\la r\ra^{\frac{1}{2}-s+\sigma} \la t-r\ra^{-\sigma}
  r^{-1}(\partial_t+\partial_r)(r Z^{\le 2}g)\Bigr\|^{p-1}_{L^2_{t,x}(\tC^U_{t,x})}\Bigr)\\\times
     \Bigl\|\la r\ra^{-\frac{1}{2}-s+\sigma}\la t-r\ra^{-\sigma} 
     f\Bigr\|_{L^2_{t,x}(C^U_\tau)}.
  \end{multline*}
Thus by H\"older's inequality and noting that $2p>2$, we see that the
left side of
\eqref{UZ2} is bounded by
\begin{multline*}
\Bigl(  \Bigl\|\la r\ra^{-\frac{1}{2}-s+\sigma} \la t-r\ra^{-\sigma} Z^{\le 2}
g\Bigr\|^{p-1}_{\ell^p_U \ell^2_{\tau\ge 4U} L^2_{t,x}(\tilde{C}^U_\tau)}
+   \Bigl\|\la r\ra^{\frac{1}{2}-s+\sigma} \la t-r\ra^{-\sigma} r^{-1}
(\partial_t+\partial_r) (r Z^{\le 2}
g)\Bigr\|^{p-1}_{\ell^p_U\ell^2_{\tau\ge 4U} L^2_{t,x}(\tC^U_\tau)}\Bigr)
\\\times
\Bigl\|\la r\ra^{-\frac{1}{2}-s+\sigma}\la t-r\ra^{-\sigma} 
     f\Bigr\|_{\ell^p_U \ell^2_{\tau\ge 4U}L^2_{t,x}(C^U_\tau)}
\end{multline*}
which is $\mathcal{O}( \|Z^{\le 2}g\|^{p-1}_X\|f\|_X)$
as desired.

\vskip .125in
\subsubsection{Proof of \eqref{RZ2}}
This proof follows the same outline as \eqref{UZ2}, but absorbing the
summations is less delicate.  Indeed, by
H\"older's inequality, we get
\[
  \|\la r\ra^{\frac{3}{2}-s} |g|^{p-1}f\|_{L^2_{t,x}(C^R_\tau)}\lesssim
  \|\la r\ra^{\frac{2}{p-1}} g\|^{p-1}_{L^\infty_{t,x}(C^R_\tau)}
  \|\la r\ra^{-\frac{1}{2}-s} f\|_{L^2_{t,x}(C^R_\tau)}.
\]
So long as $1\le R\le \tau/2$, we may apply \eqref{newCRT} to the first factor, which
provides the bound
\begin{multline*}
  \|\la r\ra^{\frac{3}{2}-s} |g|^{p-1}f\|_{L^2_{t,x}(C^R_\tau)}
  \lesssim
\|\la r\ra^{-\frac{1}{2}-s} f\|_{L^2_{t,x}(C^R_\tau)}\\\times
  \Bigl(\|\la r\ra^{-\frac{1}{2}-s} Z^{\le
    2}g\|^{p-1}_{L^{2}_{t,x}(\tC^R_\tau)}
  +\|\la
  r\ra^{\frac{1}{2}-s}r^{-1}(\partial_t+\partial_r)(rZ^{\le 2}
  g)\|^{p-1}_{L^2_{t,x}(\tC^R_\tau)}
  + \|\la r\ra^{\frac{1}{2}-s} \ang Z^{\le 2} g\|^{p-1}_{L^2_{t,x}(\tC^R_\tau)}
  \Bigr).
\end{multline*}
Summing over $R, \tau$ and using that $\|\cdot \|_{\ell^{2p}} \lesssim
\|\cdot \|_{\ell^2}$ as $p>1$ gives \eqref{RZ2}.

\vskip .125in
\subsubsection{Proof of \eqref{UZ1}}
For fixed $1\le U\le \tau/4$, we use H\"older's inequality to see that
\[
  \Bigl\|\tau^{1-s} U^{\frac{1}{2}} |g|^{p-2}f^2\Bigr\|_{L^2_{t,x}(C^U_\tau)}
  \lesssim
  \tau^{1-s}U^{\frac{1}{2}}\|g\|^{p-2}_{L^\infty_{t,x}(C^U_\tau)}
  \|f\|^2_{L^4_{t,x}(C^U_\tau)}. 
\]
To the first factor on the right side, we apply \eqref{newCUT}, and on
the second factor we use \eqref{LqCUT}.  This gives that the right
side is controlled by
\begin{multline*}
  \Bigl(\|\la r\ra^{-\frac{1}{2}-s+\sigma} \la t-r\ra^{-\sigma} Z^{\le
    2} g\|_{L^2_{t,x}(\tC^U_\tau)}
  + \|\la r\ra^{\frac{1}{2}-s+\sigma} \la t-r\ra^{-\sigma}
  r^{-1}(\partial_t+\partial_r)(rZ^{\le 2}g)\|_{L^2_{t,x}(\tC^U_\tau)}\Bigr)^{p-2}
  \\\times
  \Bigl(\|\la r\ra^{-\frac{1}{2}-s+\sigma} \la t-r\ra^{-\sigma} Z^{\le
    1} f\|_{L^2_{t,x}(\tC^U_\tau)}
  + \|\la r\ra^{\frac{1}{2}-s+\sigma} \la t-r\ra^{-\sigma}
  r^{-1}(\partial_t+\partial_r)(rZ^{\le 1} f)\|_{L^2_{t,x}(\tC^U_\tau)}\Bigr)^2.
\end{multline*}
And, thus, it follows that
\begin{multline*}
 \Bigl\|\tau^{1-s} U^{\frac{1}{2}}
 |g|^{p-2}f^2\Bigr\|_{\ell^1_U\ell^2_{\tau\ge 4U}L^2_{t,x}(C^U_\tau)}
 \\\lesssim\Bigl(\|\la r\ra^{-\frac{1}{2}-s+\sigma} \la t-r\ra^{-\sigma} Z^{\le
    2} g\|^{p-2}_{\ell^p_U \ell^2_{\tau\ge 4U}L^2_{t,x}(\tC^U_\tau)}
  + \|\la r\ra^{\frac{1}{2}-s+\sigma} \la t-r\ra^{-\sigma}
  r^{-1}(\partial_t+\partial_r)(rZ^{\le
    2}g)\|^{p-2}_{\ell^{p}_U\ell^2_{\tau\ge 4U}L^2_{t,x}(\tC^U_\tau)}\Bigr)
  \\\times
  \Bigl(\|\la r\ra^{-\frac{1}{2}-s+\sigma} \la t-r\ra^{-\sigma} Z^{\le
    1} f\|^2_{\ell^p_U\ell^2_{\tau\ge 4U}L^2_{t,x}(\tC^U_\tau)}
  + \|\la r\ra^{\frac{1}{2}-s+\sigma} \la t-r\ra^{-\sigma}
  r^{-1}(\partial_t+\partial_r)(rZ^{\le 1}
  f)\|^2_{\ell^p_U\ell^2_{\tau\ge 4U}L^2_{t,x}(\tC^U_\tau)}\Bigr),
\end{multline*}
which is $\O(\|Z^{\le 2} g\|_X^{p-2}\|Z^{\le 1} f\|_X^2)$ as needed.

\vskip .125in
\subsubsection{Proof of \eqref{RZ1}}
This follows the same strategy as the more subtle case \eqref{UZ1}.
By H\"older's inequality, \eqref{newCRT}, and \eqref{LqCRT}, we get
\begin{multline*}
  \|\la r\ra^{\frac{3}{2}-s} |g|^{p-2}f^2\|_{L^2_tL^2_x(C^R_\tau)}
  \\\lesssim
  \Bigl(\|\la r\ra^{-\frac{1}{2}-s} Z^{\le
    2}g\|_{L^2_{t,x}(\tC^R_\tau)} + \|\la r\ra^{\frac{1}{2}-s}
  r^{-1}(\partial_t+\partial_r)(r Z^{\le
    2}g)\|_{L^2_{t,x}(\tC^R_\tau)}
  +\|\la r\ra^{\frac{1}{2}-s} \ang Z^{\le
    2}g\|_{L^2_{t,x}(\tC^R_\tau)}
  \Bigr)^{p-2}
  \\\times
  \Bigl(\|\la r\ra^{-\frac{1}{2}-s} Z^{\le
    1}f\|_{L^2_{t,x}(\tC^R_\tau)}
  +\|\la r\ra^{\frac{1}{2}-s} r^{-1}(\partial_t+\partial_r)(r Z^{\le
    1}f)\|_{L^2_{t,x}(\tC^R_\tau)}
  +\|\la r\ra^{\frac{1}{2}-s} \ang Z^{\le
    1}f\|_{L^2_{t,x}(\tC^R_\tau)}\Bigr)^2
\end{multline*}
provided that $1\le R\le \tau/2$.
Summing over $R, \tau$ then yields \eqref{RZ1}.

\subsection{Boundedess}\label{section_bdd}
In this section, we use induction to show that there exists a constant
$C_0$ (independent of $j$) so that
\begin{equation}
  \label{bound_goal}
  \|Z^{\le 2} w_j\|_X \le 2C_0\varepsilon.
\end{equation}
By \eqref{rpZ2} and the fact that $w_{(0)},
w_{(1)}\in C^\infty_c(\R^3)$, there exists $C_0$ so that
\[\|Z^{\le 2}w_1\|_X \le C_0\varepsilon,\]
which establishes the base case.

We now assume that
\eqref{bound_goal} holds for $w_j$.
A subsequent application of \eqref{rpZ2} gives
\begin{multline}\label{bound}\|Z^{\le 2} w_{j+1}\|_X \le C_0 \varepsilon
  +C \|\la r\ra^{\frac{3}{2}-s} Z^{\le 2}F_p(w_j)\|_{\ell^2_\tau \ell^2_{R\le \tau/2}
    L^2_{t,x}(C^R_\tau)} \\+ C\|\la r\ra^{1-s} \la t-r\ra^{\frac{1}{2}}
    Z^{\le 2}F_p(w_j)\|_{\ell^1_U \ell^2_{\tau\ge 4U} L^2_{t,x}(C^U_\tau)}.
  \end{multline}
  By \eqref{Fp}, we have
  \begin{multline*}
    \|\la r\ra^{\frac{3}{2}-s} Z^{\le 2}F_p(w_j)\|_{\ell^2_\tau \ell^2_{R\le \tau/2}
      L^2_{t,x}(C^R_\tau)}
    \lesssim \|\la r\ra^{\frac{3}{2}-s} |w_j|^{p-1} Z^{\le 2}w_j\|_{\ell^2_\tau \ell^2_{R\le \tau/2}
      L^2_{t,x}(C^R_\tau)}
    \\+\|\la r\ra^{\frac{3}{2}-s} |w_j|^{p-2} (Z^{\le 1} w_j)^2\|_{\ell^2_\tau \ell^2_{R\le \tau/2}
    L^2_{t,x}(C^R_\tau)},
  \end{multline*}
which is $\O(\|Z^{\le 2} w_j\|_X^p)$ by \eqref{RZ2} and \eqref{RZ1}.
In a similar fashion
\begin{multline*}
  \|\la r\ra^{1-s} \la t-r\ra^{\frac{1}{2}}
  Z^{\le 2}F_p(w_j)\|_{\ell^1_U \ell^2_{\tau\ge 4U}
      L^2_{t,x}(C^U_\tau)}
  \lesssim
    \|\la r\ra^{1-s} \la t-r\ra^{\frac{1}{2}}
    |w_j|^{p-1} Z^{\le 2}w_j\|_{\ell^1_U \ell^2_{\tau\ge 4U}
      L^2_{t,x}(C^U_\tau)}
    \\+\|\la r\ra^{1-s} \la t-|x|\ra^{\frac{1}{2}}
    |w_j|^{p-1}(Z^{\le 1}w_j)^2\|_{\ell^1_U \ell^2_{\tau\ge 4U} L^2_{t,x}(C^U_\tau)},
\end{multline*}
for which the $\O(\|Z^{\le 2} w_j\|^p_X)$ bound follows from
\eqref{UZ2} and \eqref{UZ1}.  Using these estimates in \eqref{bound}
gives
\[\|Z^{\le 2} w_{j+1}\|_X \le C_0\varepsilon + C\|Z^{\le 2}w_j\|^p_X.\]
Applying the inductive hypothesis and choosing $\varepsilon$
sufficiently small then establishes \eqref{bound_goal}.

\subsection{Convergence}\label{section_Cauchy}
We next consider $\|w_{j+1}-w_j\|_X$ and shall show that
\begin{equation}
  \label{Cauchy}
  \|w_{j+1}-w_j\|_X \le \frac{1}{2}\|w_j-w_{j-1}\|_X,\quad \text{ for
    all } j\ge 1.
\end{equation}
This suffices to prove that $(w_j)$ is Cauchy in $X$, and thus
convergent, which will complete the proof of Theorem \ref{main_thm}.

Due to \eqref{Fp} and the Mean Value Theorem,
\[|F_p(w_{j})-F_p(w_{j-1})|\lesssim \Bigl(|w_j|^{p-1} +
  |w_{j-1}|^{p-1}\Bigr) |w_j-w_{j-1}|.\]
Thus, by \eqref{RZ2} and \eqref{UZ2},
\[
  \|\la r\ra^{\frac{3}{2}-s} (F_p(w_j)-F_p(w_{j-1}))\|_{\ell^2_\tau \ell^2_{R\le \tau/2}
    L^2_{t,x}(C^R_\tau)}
  \lesssim \Bigl(\|Z^{\le 2}w_{j-1}\|_X^{p-1}+\|Z^{\le 2}w_j\|^{p-1}_X\Bigr)\|w_j-w_{j-1}\|_X
\]
and
\[
  \|\la r\ra^{1-s} \la t-r\ra^{\frac{1}{2}}
(F_p(w_j)-F_p(w_{j-1})\|_{\ell^1_U \ell^1_{\tau\ge 4U}
      L^2_{t,x}(C^U_\tau)}  \lesssim \Bigl(\|Z^{\le 2}w_{j-1}\|_X^{p-1}+\|Z^{\le 2}w_j\|^{p-1}_X\Bigr)\|w_j-w_{j-1}\|_X.
\]
An application of \eqref{rp} shows
\[\|w_{j+1}-w_j\|_X \lesssim  \Bigl(\|Z^{\le 2}w_{j-1}\|_X^{p-1}+\|Z^{\le 2}w_j\|^{p-1}_X\Bigr)\|w_j-w_{j-1}\|_X.\] 
Using \eqref{bound_goal} for $\varepsilon$ sufficiently small then
yields \eqref{Cauchy} as desired.



\end{document}